\documentclass[a4paper,12pt]{scrartcl}

\usepackage[utf8]{inputenc}

\usepackage{amsthm,amsmath,amsfonts,amssymb}
\numberwithin{equation}{section} 

\usepackage{mathrsfs}

\usepackage[final]{showlabels} 

\showlabels{cite} \showlabels{bibitem} \showlabels{ref} \showlabels{eqref}

\usepackage{mathtools}	
\mathtoolsset{showonlyrefs=true} 
\usepackage{bm}	
\usepackage{bbm}	
\usepackage{bbold}

\usepackage{tensor}	

\usepackage[colorlinks=true]{hyperref}	
\usepackage[all]{hypcap}      

\newcommand{\m}[1]{\mathcal{#1}}
\newcommand{\bb}[1]{\mathbb{#1}}

\newcommand{\mrm}[1]{\mathrm{#1}}
\newcommand{\mscr}[1]{\mathscr{#1}}
\newcommand{\f}[1]{\mathfrak{#1}}

\newcommand{\diff}{\partial}
\newcommand{\bdiff}{\bar{\partial}} 

\newcommand{\I}{\mathrm{i}\mkern1mu}	
\newcommand{\transpose}{\intercal}		

\DeclarePairedDelimiter\abs{\lvert}{\rvert}	
\DeclarePairedDelimiter\norm{\lVert}{\rVert}	

\DeclarePairedDelimiter{\set}{\{}{\}}	
\newcommand{\tc}{\mathrel{}\mathclose{}\middle|\mathopen{}\mathrel{}}	

\newcommand{\cotJ}{T^*\!\m{J}}	

\newtheorem{thm}{Theorem}[section]

\newtheorem{prop}[thm]{Proposition}
\newtheorem{lemma}[thm]{Lemma}
\newtheorem{cor}[thm]{Corollary}
\newtheorem{conj}[thm]{Conjecture}

\theoremstyle{definition}
\newtheorem{definition}[thm]{Definition}

\theoremstyle{remark}

\newtheorem{rmk}[thm]{Remark}

\title{Scalar curvature and deformations of complex structures}
\author{Carlo Scarpa}
\date{}

\begin{document}

\maketitle

\begin{abstract}
\noindent\textsc{Abstract.} We study a system of equations on a compact complex manifold, that couples the scalar curvature of a K\"ahler metric with a spectral function of a first-order deformation of the complex structure. The system comes from an infinite-dimensional K\"ahler reduction, which is a hyperk\"ahler reduction for a particular choice of the spectral function. The system can be formally complexified using a flat connection on the space of first-order deformations that are compatible with a K\"ahler metric. We describe a variational characterization of the equations, a Futaki invariant for the system, and a generalization of K-stability that is conjectured to characterize the existence of solutions to the system. We verify a particular case of this conjecture in the context of toric manifolds.
\end{abstract}


\section{Introduction}

Consider a compact symplectic manifold~$(M,\omega)$ of dimension~$2n$, and let~$\m{J}$ be the set of all (integrable) complex structures that are compatible with~$\omega$. It is the space of sections of a bundle on~$M$, with fibres isomorphic to the Hermitian symmetric space~$\mrm{Sp}(2n)/U(n)$, so it is an infinite-dimensional locally symmetric K\"ahler manifold. Our interest in this space of complex structures comes from the fact, first noted by Fujiki~\cite{Fujiki_moduli} and Donaldson~\cite{Donaldson_scalar}, that the group~$\m{G}$ of exact symplectomorphisms of~$(M,\omega)$ acts by pullbacks on~$\m{J}$, and this action is Hamiltonian. A moment map is given by assigning to each complex structure~$J$ the scalar curvature of the K\"ahler metric~$g_J$ defined by~$\omega$ and~$J$.

In this work, we extend this picture by considering the total space of the cotangent bundle~$\cotJ$. A typical element of~$\cotJ$ is given by a pair of a complex structure~$J\in\m{J}$ and a first-order deformation~$\alpha$ of~$J$, that is a form~$\alpha\in\m{A}^{0,1}_J(T^{1,0}_JM)$ satisfying the integrability condition~$\bdiff_J\alpha=0$. Of course,~$\alpha$ must also satisfy a compatibility condition with the symplectic form: the contraction~$g_J(\alpha\cdot,\cdot)\in\Gamma(T^{0,1}_JM\otimes T^{0,1}_JM)$ must be a symmetric tensor.

Since~$\m{J}$ is a K\"ahler manifold,~$\cotJ$ has a canonical complex structure.  We can consider Calabi ansatz metrics on~$\cotJ$
\begin{equation}\label{eq:metrica_TJ}
\bm{\Omega}=\pi^*\Omega+\mrm{d}\mrm{d}^cF
\end{equation}
where~$\pi^*\Omega$ is the pullback to~$\cotJ$ of the Fujiki-Donaldson K\"ahler form on~$\m{J}$, and~$F$ is a function of the shape
\begin{equation}\label{eq:metrica_TJ_f}
F(J,\alpha)=\int_Mf(\bar{\alpha}\alpha)\frac{\omega^n}{n!}
\end{equation}
for some convex spectral function~$f$ of the endomorphism~$\bar{\alpha}\alpha\in\mrm{End}({T^{1,0}_J}^*M)$, which is self-adjoint for the Hermitian product defined by~$g_J$. If~$f$ is convex the~$(1,1)$-form in~\eqref{eq:metrica_TJ} will be positive, and it coincides with the Donaldson-Fujiki form on the zero-section of~$\cotJ$, see Appendix~\ref{sec:geometry_TJ}.

The pullback action of~$\m{G}$ on~$\m{J}$ lifts to a pullback action on~$\cotJ$, which preserves the K\"ahler form~$\bm{\Omega}$. As the metric~\eqref{eq:metrica_TJ} differs from the Donaldson-Fujiki metric by an exact~$2$-form, the action of~$\m{G}$ is Hamiltonian (see~\cite[Lemma~$3.5$]{ScarpaStoppa_hyperk_reduction}), and if we identify the Lie algebra of~$\m{G}$ with~$\m{C}^\infty(M)$ a moment map is given by
\begin{equation}\label{eq:mm_real_definition}
\left\langle\f{m}_{\bm{\Omega}}(J,\alpha),\varphi\right\rangle=\int_M2\left(s(\omega,J)-\hat{s}\right)\varphi\frac{\omega^n}{n!}+\int_M\mrm{d}^cf_{J,\alpha}\left(\m{L}_{X^\omega_\varphi}J,\m{L}_{X^\omega_\varphi}\alpha\right)\frac{\omega^n}{n!}
\end{equation}
where we are considering~$f$ as a function on~$T^*\left(\mrm{Sp}(2n)/U(n)\right)$ and~$\mrm{d}^c$ is the twisted differential of this space. We refer to Appendix~\ref{sec:geometry_TJ} for more details on the interplay between~$\cotJ$ and~$T^*\left(\mrm{Sp}(2n)/U(n)\right)$. We can integrate by parts the second term on the right-hand side of~\eqref{eq:mm_real_definition} to express it as~$L^2$-pairing of a function with~$\varphi$. The moment map is then identified with the function
\begin{equation}\label{eq:mm_real_explicit}
\f{m}_{\bm{\Omega}}(J,\alpha)=s(\omega,J)-\hat{s}+\mrm{div}\,\Re\left(\I Jg_J(\nabla^\sharp\bar{\alpha},\alpha f'(\bar{\alpha}\alpha))+2\nabla^*(\bar{\alpha}\alpha f'(\bar{\alpha}\alpha))\right).
\end{equation}
Here~$f'(\bar{\alpha}\alpha)$ is shorthand for the gradient of the spectral function~$f$: in a system of normal K\"ahler coordinates for~$g_J$ around a point of~$M$, one can write~$\bar{\alpha}\alpha=U\Lambda U^*$ for some unitary matrix~$U$ and a diagonal, positive matrix~$\Lambda=(\lambda_1,\dots,\lambda_n)$. Then, one has
\begin{equation}
f'(\bar{\alpha}\alpha)=U\mrm{diag}\left(\nabla f(\lambda_1,\dots,\lambda_n)\right)U^*.
\end{equation}

Notice that~$\cotJ$ also carries a canonical holomorphic-symplectic form~$\bm{\Theta}$, given by the differential of the tautological~$1$-form on the cotangent bundle. It is invariant under the~$\m{G}$-action on~$\cotJ$, and the action is also Hamiltonian with respect to~$\bm{\Theta}$. The moment map is given by
\begin{equation}\label{eq:complex_mm}
\f{m}_{\bm{\Theta}}(J,\alpha)=\m{D}^*\alpha
\end{equation}
where~$\m{D}^*$ is the formal adjoint of the Lichnerowicz operator defined by the K\"ahler metric~$g_J$. For a K\"ahler metric~$\omega$, the space of solutions~$\alpha$ of~\eqref{eq:complex_mm} belonging to a fixed class in~$H^1(TM)$ is finite-dimensional, see Lemma~\ref{lem:dim_complex_mm_sols}.

In this paper, we study the system of moment map equations
\begin{equation}\label{eq:system_mm_implicit}
\begin{cases}
\f{m}_{\bm{\Theta}}(J,\alpha)=0\\
\f{m}_{\bm{\Omega}}(J,\alpha)=0
\end{cases}
\end{equation}
for any choice of K\"ahler form~$\bm{\Omega}$ coming from the ansatz~\eqref{eq:metrica_TJ}. Both equations come from an infinite-dimensional K\"ahler reduction. We expect that the space of solutions of~\eqref{eq:system_mm_implicit} can be used to study the (hypothetical) moduli space of polarized manifolds together with a class of first-order deformations of the complex structure. The first equation in~\eqref{eq:system_mm_implicit} is of independent interest: the space of integrable first-order deformations of a complex structure that are compatible with a symplectic form and solve the equation~$\m{D}^*\alpha=0$ is a key ingredient in the deformation theory of constant scalar curvature K\"ahler (cscK) manifolds and the construction of their moduli spaces, see~\cite{FujikiSchumacher_moduli},~\cite{Szekelyhidi_moduli}, and~\cite{DervanNaumann_moduli}.

A particular case of the system~\eqref{eq:system_mm_implicit} is the \emph{Hitchin-cscK system}, introduced in~\cite{ScarpaStoppa_hyperk_reduction}, which is our main motivation for studying~\eqref{eq:system_mm_implicit}. For a particular choice of~$f$ in~\eqref{eq:metrica_TJ_f}, the three symplectic forms~$\left(\bm{\Omega},\Re\bm{\Theta},\Im\bm{\Theta}\right)$ define a hyperk\"ahler structure on~$\cotJ$. The unique function with this property was found in~\cite{Biquard_Gauduchon} in the context of defining hyperk\"ahler structures on the cotangent bundle of Hermitian symmetric spaces. Denoting by~$\lambda_1,\dots,\lambda_n$ the eigenvalues of~$\bar{\alpha}\alpha$, the potential that defines a hyperk\"ahler metric is
\begin{equation}\label{eq:hyperk_potential}
f(\bar{\alpha}\alpha)=\sum_{a=1}^n1-\sqrt{1-\lambda_a}+\log\frac{1+\sqrt{1-\lambda_a}}{2}.
\end{equation}
Notice however that~\eqref{eq:hyperk_potential} is smooth only in the locus~$\m{U}\subseteq\cotJ$ where the eigenvalues of~$\bar{\alpha}\alpha$ are smaller than~$1$: indeed, the hyperk\"ahler metric defined by this potential is well-defined only on~$\m{U}$, and is incomplete. In the present paper, we will only focus on the case when~$\bm{\Omega}$ is defined on the whole cotangent space, but most of the results still hold in the more general case where we consider just an open subset of~$\cotJ$.

We give an overview of the main results of each section.

\paragraph*{The complexified system.} The variables in~\eqref{eq:system_mm_implicit} are a complex structure~$J$ and a first-order deformation~$\alpha$ of~$J$. From the point of view of complex geometry, however, it is more natural to consider the complex structure of~$M$ to be fixed, and to let~$\omega$ and~$\alpha$ vary in the K\"ahler class~$[\omega]\in H^{1,1}(M)$ and the deformation class~$[\alpha]\in H^{1}(TM)$ instead. In Section~\ref{sec:complexification_eqs} we interpret this change of viewpoint as a \emph{formal complexification} of the action~$\m{G}\curvearrowright\cotJ$. This is well understood for the cscK equation, but the presence of the first-order deformation~$\alpha$ in our system of equations introduces some complications. The main issue is that if~$\alpha$ is compatible with a K\"ahler form, it is in general not compatible with other forms in the same K\"ahler class.

We study in Section~\ref{sec:bundle_class} the space~$\m{E}$ of compatible pairs~$(\alpha,\omega)$, which is a vector bundle on the K\"ahler class. The crucial observation is that we can define a connection on~$\m{E}$ by
\begin{equation}\label{eq:connection_bundle}
\left(D_\varphi\alpha\right)(\omega)=\diff_{t=0}\alpha(\omega+\I\diff\bdiff t\varphi)-\bdiff\left(\alpha(\diff\varphi)^{\sharp_{g_J}}\right)
\end{equation}
and this connection is flat, see Proposition~\ref{prop:flat}. Hence, any first-order deformation~$\alpha$ compatible with a K\"ahler form~$\omega_0$ can be uniquely extended to a horizontal section of~$\m{E}$, which to each K\"ahler form~$\omega\in[\omega_0]$ assigns a compatible first order deformation~$\alpha(\omega)$. These horizontal sections of~$\m{E}$ play a fundamental role in the formal complexification of the action~$\m{G}\curvearrowright\cotJ$.
\begin{thm}\label{thm:intro_complexification}
Fix~$(J,\alpha)$ compatible with~$\omega_0$, and extend~$\alpha$ to a~$D$-horizontal section of~$\m{E}$. The system of equations~\eqref{eq:system_mm_implicit} along the complexified orbit of~$(J,\alpha)$ is equivalent to the following system of equations for a K\"ahler metric~$\omega\in[\omega_0]$
\begin{equation}\label{eq:complessificate_intro}
\begin{dcases}
\f{m}_{\bm{\Theta}}(\omega,J,\alpha(\omega))=0\\
\f{m}_{\bm{\Omega}}(\omega,J,\alpha(\omega))=0.
\end{dcases}
\end{equation}
\end{thm}

\paragraph*{Features of the complexified equations.} Together with the moment map interpretation of the equations, the formal complexification of the action and the system of equations described in Section~\ref{sec:complexification} has many interesting consequences. In Section~\ref{sec:energy} we will show that equation~\eqref{eq:mm_real_explicit} (more precisely, its complexified version) is the Euler-Lagrange equation for a functional on the K\"ahler class. As for the first equation in~\eqref{eq:complessificate_intro}, we show in Section~\ref{sec:complex_eq}  
\begin{prop}\label{prop:intro_sol_complex_mm}
If~$(J,\alpha)\in\cotJ$ is a solution of the complex moment map equation~$\m{D}^*\alpha=0$, the complexified orbit of~$(J,\alpha)$ lies in the zero-locus of~$\m{D}^*$.
\end{prop}
In Lemma~\ref{lem:dim_complex_mm_sols} we study the existence of solutions of~$\m{D}^*\alpha=0$. The result is particularly interesting in the polarized case, when we consider an ample line bundle~$L\to M$ and K\"ahler metrics in~$\mrm{c}_1(L)$. For each K\"ahler form~$\omega\in\mrm{c}_1(L)$, the equation~$\m{D}^*\alpha=0$ gives a way to choose a canonical representative of any first order deformation of the pair~$(M,L)$ in the sense of~\cite[\S$3.3.3$]{Sernesi_deformations}, see Remark~\ref{rmk:deformation_polarized}. Hence, solutions of the system~\eqref{eq:complessificate_intro} are canonical representatives of the K\"ahler class~$\mrm{c}_1(L)$ and a first order deformation class in~$H^1(M,\m{E}_L)$, where~$\m{E}_L$ is the Atiyah extension of~$L$, the element associated to~$\mrm{c}_1(L)$ under the identification~$H^1(M,T^*{}^{1,0}M)\cong\mrm{Ext}^1_{\m{O}_M}(T^{1,0}M,\m{O}_M)$.

From Section~\ref{sec:linearization} onwards, we focus on spectral functions~$f(\bar{\alpha}\alpha)$ of the shape
\begin{equation}\label{eq:spectral_k}
f(\bar{\alpha}\alpha)=\sum_ak(\lambda_a)
\end{equation}
for a convex, non-decreasing function~$k$, where~$\lambda_1,\dots,\lambda_n$ are the eigenvalues of~$\bar{\alpha}\alpha$.
\begin{thm}\label{thm:intro_linearizzazione}
For a Calabi ansatz metric~$\bm{\Omega}$ defined by the spectral function~\eqref{eq:spectral_k}, the linearization of~$\f{m}_{\bm{\Omega}}(\omega,\alpha(\omega))$ is elliptic. The linearization around a solution is self-adjoint, and the kernel is in~$1$-to-$1$ correspondence with the functions~$\varphi\in\m{C}^\infty(M,\bb{R})$ such that~$\m{L}_{X^\omega_\varphi}J=0$ and~$\m{L}_{X^\omega_\varphi}\alpha=0$.
\end{thm}
As a consequence, if~$(\omega,\alpha)$ is a solution of the (complexified) system~\eqref{eq:complessificate_intro} and~$M$ has discrete reduced automorphism group, then we can slightly deform~$\alpha$ in any direction and still obtain solutions of~\eqref{eq:complessificate_intro}.

\paragraph*{The Futaki Invariant and K-stability.} In Section~\ref{sec:Futaki} we study an obstruction to the existence of solutions of the real moment map equation. The Futaki invariant for the moment map equation~$\f{m}_{\bm{\Omega}}=0$ actually coincides with the classical one, except for the fact that it is defined on a sub-algebra~$\f{h}_\alpha$ of the Lie algebra~$\f{h}_0$ of the holomorphic vector fields admitting a holomorphy potential (see Section~\ref{sec:Futaki} for the definition of~$\f{h}_\alpha$). In general,~$\f{h}_\alpha$ will be properly contained in~$\f{h}_0$.
\begin{prop}\label{prop:Futaki}
If there exists a solution of~\eqref{eq:mm_real_definition} in the K\"ahler class~$[\omega_0]$ then the Futaki invariant of~$[\omega_0]$ vanishes on~$\f{h}_\alpha$.
\end{prop}
Hence, the obstruction to the existence of solutions of~\eqref{eq:complessificate_intro} seem very close to those of the cscK equation. This phenomenon is similar to a well-known property of the Hermitian Yang-Mills equation and the Higgs bundle equations introduced by Hitchin: the existence of solutions to the HYM equation on a holomorphic bundle~$E\to M$ is equivalent to slope-stability of~$E$, a condition that must be checked on all proper subbundles. To characterize the existence of solutions to the Higgs bundle equations one can use the \emph{same} condition, which however needs to be checked just on a smaller set of subbundles, those that are compatible with the Higgs field. The Higgs bundle equations can be seen as an infinite-dimensional hyperk\"ahler reduction of the cotangent space of holomorphic structures on~$E$, so the Hitchin-cscK system obtained from~\eqref{eq:system_mm_implicit} by choosing~$f$ as in~\eqref{eq:hyperk_potential} is a direct analogue of the Higgs bundle equations in the category of polarized varieties.

Motivated by this analogy with Higgs bundles, in Section~\ref{sec:stability} we describe a version of K-stability that should characterize the existence of solutions to our system~\eqref{eq:complessificate_intro}. The idea is to identify, among all test configurations for a polarized variety~$(X,L)$, a suitable class of test configurations that are compatible with a deformation class~$\eta\in H^1(X,\m{E}_L)$. Briefly, a \emph{compatible test configuration} is a test configuration~$(\m{X},\m{L})$ for~$(X,L)$ together with a first-order deformation of the fibration~$\m{X}\to\bb{C}$ (as in~\cite{Horikawa_deformations1, Horikawa_deformations2}) that is~$\bb{C}^*$-invariant, compatible with the relative polarization~$\m{L}$, and restricts to~$\alpha$ on the general fibre. In other words, we ask all the fibres of the test configuration to admit a deformation ``in the same direction'' as~$\eta$. This condition on~$(\m{X},\m{L})$ is trivial outside of the central fibre, so the problem is to find conditions that guarantee that the deformation extends to the whole family. It seems natural to pose the following
\begin{conj}\label{conj:intro_Kstab}
Let~$(X,L)$ be a polarized complex variety, and let~$\eta\in H^1(X,\m{E}_L)$ be a deformation class of~$(X,L)$. Then~\eqref{eq:complessificate_intro} admits a solution~$\omega\in\mrm{c}_1(L)$,~$\alpha(\omega)\in\eta$, if and only if~$(X,L)$ is K-stable with respect to test configurations that are compatible with~$\eta$.
\end{conj}
We refer to Section~\ref{sec:stability} for a more detailed discussion of this stability condition, which we formulate more precisely for an arbitrary K\"ahler manifold, following~\cite{DervanRoss_Kstab}.

\paragraph*{The toric case.} In Section~\ref{sec:toric}, we study the moment map equation~\eqref{eq:mm_real_explicit} on a toric manifold. Rather than studying the complexified equation, we fix a torus invariant symplectic form, and we look for torus-invariant complex structures and first-order deformations that solve~\eqref{eq:mm_real_explicit}. We will see, however, that the description in symplectic coordinates of deformations of complex structures implicitly parametrize a complexified orbit of the action~$\m{G}\curvearrowright\cotJ$.

Using this description of the complexification, we generalize some of the results already appeared in~\cite{ScarpaStoppa_symplectic} to the case of arbitrary spectral functions in the definition of~\eqref{eq:metrica_TJ}, verifying a special case of Conjecture~\ref{conj:intro_Kstab}. Notice that, since we only consider torus-invariant deformations~$\alpha$, the integrability condition together with the moment map equation~$\m{D}^*\alpha=0$ imply~$\alpha=0$. Hence, in the toric case we just study the real moment map equation \eqref{eq:mm_real_explicit}, rather than the full system \eqref{eq:complessificate_intro}.
\begin{thm}\label{thm:intro_toric_existence}
Let~$M$ be a toric manifold. For any small enough torus-invariant deformation of the complex structure, there exists a torus-invariant solution of~\eqref{eq:mm_real_explicit} if and only if~$M$ is (toric) uniformly K-stable.
\end{thm}

\paragraph*{Acknowledgements.} The author is grateful to Jacopo Stoppa and Ruadha\'i Dervan for many discussions and useful suggestions on an earlier version of this paper.

\section{The complexified system}\label{sec:complexifiedsystem}

\subsection{A vector bundle on the K\"ahler class}\label{sec:bundle_class}

Fix an integrable complex structure~$J$, and let~$\Omega$ be a K\"ahler class on~$(M,J)$. We define a vector bundle~$\m{E}\to\Omega$ as follows: consider the set~$\m{A}$ of all first-order integrable deformations of~$J$
\begin{equation}
\m{A}=\set*{\alpha\in\m{A}^{0,1}(T^{1,0}M)\tc\bdiff\alpha=0}
\end{equation}
and denote by~$\m{E}\subseteq\m{A}\times\Omega$ the set of all compatible pairs~$(\alpha,\omega)$, where~$\alpha$ and~$\omega$ are compatible if and only if~$g_J(\alpha\cdot,\cdot)\in\Gamma(T^*M\otimes T^*M)$ is symmetric.

Together with the projection on the second component,~$\m{E}$ is a vector bundle on the K\"ahler class. For~$\omega\in\Omega$, we denote by~$\m{E}_\omega$ fibre over~$\omega$ of~$\m{E}\to\Omega$, so that~$\m{E}_\omega$ is the set of all integrable first-order deformations that are compatible with~$\omega$.

\begin{lemma}\label{lemma:compatibility_class}
Fix a K\"ahler form~$\omega\in\Omega$. A first-order deformation class~$\eta\in H^1(T^{1,0}M)$ intersects~$\m{E}_\omega$ if and only if~$\eta\cdot\Omega=0$, where~$\cdot$ denotes the composition of cup product with contraction
\begin{equation}\label{eq:cohom_product}
H^1(M,T^{1,0}M)\times H^1(M,T^*{}^{1,0}M)\xrightarrow{\cup} H^2(M,T^{1,0}M\times T^*{}^{1,0}M)\xrightarrow{c}H^2(M).
\end{equation}
In particular, a deformation class either intersects every fibre of~$\m{E}\to\Omega$ or it does not intersect any.
\end{lemma}
\begin{proof}
Fix~$\alpha\in\eta$ and~$\omega\in\Omega$. The compatibility condition between~$\alpha$ and~$\omega$ can be restated as~$\omega_{a\bar{b}}\alpha\indices{_{\bar{c}}^a}\mrm{d}\bar{z}^b\wedge\mrm{d}\bar{z}^c=0$, while~$\eta\cdot\Omega=0$ if and only if there is a~$(0,1)$-form~$\vartheta$ such that
\begin{equation}
\omega_{a\bar{b}}\alpha\indices{_{\bar{c}}^a}\mrm{d}\bar{z}^b\wedge\mrm{d}\bar{z}^c=\bdiff\vartheta.
\end{equation}
In this case, we see that~$\alpha'\coloneqq\alpha+\bdiff\left(\I\vartheta^\sharp\right)$ is compatible with~$\omega$. The same computation shows the vice-versa: if~$\alpha+\bdiff V$ is compatible with~$\omega$ for some~$(1,0)$-vector field~$V$, then
\begin{equation}
\omega_{a\bar{b}}\alpha\indices{_{\bar{c}}^a}\mrm{d}\bar{z}^b\wedge\mrm{d}\bar{z}^c=-\bdiff\left(V\lrcorner\omega\right).\qedhere
\end{equation}
\end{proof}
\begin{rmk}\label{rmk:def_polarized}
In the case when~$\Omega$ is an integral class,~$\Omega=\mrm{c}_1(L)$ for some ample line bundle~$L\to M$, the condition~$\eta\cdot\Omega=0$ is equivalent to the existence of a deformation of the pair~$(M,L)$ covering the deformation of~$M$ induced by~$\eta$, see for example \cite[Theorem~$3.3.11$]{Sernesi_deformations}.
\end{rmk}
We consider on~$\m{E}$ the connection defined in~\eqref{eq:connection_bundle}, which defines a horizontal distribution~$\m{H}$ on~$\m{E}$. Identifying~$T\Omega$ with~$\Omega\times\m{C}^\infty(M)$ through the~$\diff\bdiff$-Lemma, the distribution at a point~$(\alpha,\omega)\in\m{E}$ is
\begin{equation}\label{eq:hor_distr}
\m{H}_{\alpha,\omega}=\set*{\left(\bdiff\left(\alpha(\diff\varphi)^{\sharp_\omega}\right),\varphi\right)\tc\varphi\in\m{C}^\infty(M)}.
\end{equation}
\begin{lemma}\label{lemma:horiz_distribution}
Equation~\eqref{eq:hor_distr} describes a horizontal distribution on the bundle~$\m{E}\to\Omega$.
\end{lemma}
\begin{proof}
To prove that~$\m{H}\subset T\m{E}$, we need to show that, for every~$\omega\in\Omega$ and every~$\alpha\in\m{E}_\omega$, the~$1$-form~$\alpha+\varepsilon\,\bdiff\left(\alpha(\diff\varphi)^\sharp\right)$ is compatible with~$\omega+\varepsilon\,\I\diff\bdiff\varphi$, to first order in~$\varepsilon$. This amounts to showing that
\begin{equation}
\diff_{\bar{b}}\left(\alpha\indices{_{\bar{c}}^d}\diff_d\varphi g^{a\bar{c}}\right)g_{a\bar{e}}+\alpha\indices{_{\bar{b}}^a}\diff_a\diff_{\bar{e}}\varphi
\end{equation}
is symmetric in the indices~$\bar{b}$,~$\bar{e}$. Expanding the derivative in the first term, we get
\begin{equation}
\begin{split}
\diff_{\bar{b}}\alpha\indices{_{\bar{e}}^d}&\,\diff_d\varphi
+\alpha\indices{_{\bar{e}}^d}\,\diff_{\bar{b}}\diff_d\varphi
+\alpha\indices{_{\bar{c}}^d}\diff_d\varphi\diff_{\bar{b}}g^{a\bar{c}}\,g_{a\bar{e}}
+\alpha\indices{_{\bar{b}}^a}\diff_a\diff_{\bar{e}}\varphi=\\
=&\diff_{\bar{b}}\alpha\indices{_{\bar{e}}^d}\,\diff_d\varphi+\alpha\indices{_{\bar{e}}^d}\,\diff_{\bar{b}}\diff_d\varphi+\alpha\indices{_{\bar{b}}^d}\diff_{\bar{e}}\diff_d\varphi-\alpha\indices{_{\bar{c}}^d}\diff_d\varphi \diff_{\bar{b}}g_{q\bar{e}}\,g^{q\bar{c}}.
\end{split}
\end{equation}
The first term is symmetric in~$\bar{b}$ and~$\bar{e}$, as~$\alpha$ is integrable; the second and third are the symmetrization of~$\alpha\diff\bdiff\varphi$, while the fourth is symmetric as~$\omega$ is K\"ahler. This shows~$\m{H}\subset T\m{E}$; to prove that~$T\m{E}=\m{H}\oplus TV\m{E}$, consider a tangent vector~$(\dot{\alpha},\varphi)\in T_{\alpha,\omega}\m{E}$; as both~$\alpha+\varepsilon\,\dot{\alpha}$ and~$\alpha+\varepsilon\,\bdiff\left(\alpha(\diff\varphi)^\sharp\right)$ are compatible with~$\omega+\varepsilon\,\I\diff\bdiff\varphi$ to first order in~$\varepsilon$,~$\dot{\alpha}-\bdiff\left(\alpha(\diff\varphi)^\sharp\right)$ is compatible with~$\omega$, so that~$\left(\dot{\alpha}-\bdiff\left(\alpha(\diff\varphi)^\sharp\right),0\right)$ is a vertical tangent vector to~$\m{E}$.
\end{proof}
Let~$\omega_t=\omega_0+\I\diff\bdiff h_t$ be an analytic path in~$\Omega$, and fix~$\alpha\in\m{E}_{\omega_0}$. The unique~$D$-horizontal lift~$\alpha_t$ of~$\omega_t$ starting from~$\alpha$ is given by the solution~$\alpha_t$ of the parabolic equation
\begin{equation}\label{eq:hor_lift}
\begin{dcases}
\diff_t\alpha_t=\bdiff\left(\alpha_t(\diff h'_t)^{\sharp_t}\right)\\
\alpha_0=\alpha
\end{dcases}
\end{equation}
A solution of~$\eqref{eq:hor_lift}$ is a path of forms~$\alpha_t$. It is easy to check that each~$\alpha_t$ is compatible with the respective metric~$\omega_t$, either from Lemma \ref{lemma:horiz_distribution} or through direct computation, since~$\alpha_0$ is compatible with~$\omega_0$ and \eqref{eq:hor_lift} implies that~$\diff_t\left(g_t(\alpha_t\cdot,\cdot)-g_t(\cdot,\alpha_t\cdot)\right)\equiv 0$.

The method of characteristics shows that equation \eqref{eq:hor_lift} can be solved by a smooth~$\alpha_t$, but it is difficult to find an explicit expression for the horizontal lift of a path in~$\Omega$. In Section~\ref{sec:toric} we will give a solution in the context of toric manifolds and torus-invariant deformations of complex structures. There is a case in which the lift is trivial: if~$\alpha(\diff h)=0$, then~$\alpha$ is compatible with all the symplectic forms~$\omega_t\coloneqq \omega_0+\I\diff\bdiff\,th$, and the lift of this path in~$\Omega$ is the constant path~$\alpha_t\equiv\alpha$. Notice however that in general, even if~$\alpha$ is already compatible with~$\omega_0+\I\diff\bdiff h$, i.e.~$\bdiff\left(\alpha\left(\diff h\right)\right)=0$, we still have to move~$\alpha$ along a non-constant path~$\alpha_t$.
\begin{rmk}\label{rmk:hor_lift_class}
For any~$\omega\in\Omega$ any first-order deformation~$\alpha\in\m{E}_\omega$ defines an infinitesimal deformation class~$[\alpha]\in H^1(T^{1,0}M)$. If~$\alpha_t$ is the horizontal lift of a path~$\set{\omega_t}\subset\Omega$, the class~$[\alpha_t]\in H^1(T^{1,0}X)$ is constant: indeed, \eqref{eq:hor_lift} shows that~$[\alpha_0-\alpha_t]\equiv 0$ in~$H^1(T^{1,0}X)$.
\end{rmk}

\begin{prop}\label{prop:flat}
The connection~$D$ defined by~\eqref{eq:connection_bundle} is flat.
\end{prop}
\begin{proof}
Fix~$\alpha\in\Gamma(\Omega,\m{E})$ and~$\omega_0\in\Omega$. Let~$\varphi$ and~$\psi$ be smooth functions on~$M$, and notice that they define constant vector fields on~$\Omega$. To show that the curvature of~$D$ vanishes, it is then sufficient to check that~$D_\varphi D_\psi\alpha$ is symmetric in~$\varphi$ and~$\psi$.

As~$D_\varphi\alpha(\omega)=\varphi_\omega(\alpha)-\bdiff\left(\alpha_\omega(\diff\varphi)^{\sharp_\omega}\right)$, we get
\begin{equation}\label{eq:second_cov}
\begin{split}
D_\psi D_\varphi\alpha(\omega_0)=&\psi_{\omega_0}\left[\varphi(\alpha)-\bdiff\left(\alpha(\diff\varphi)^\sharp\right)\right]-\bdiff\left[\varphi_{\omega_0}(\alpha)-\left(\bdiff\left(\alpha_{\omega_0}(\diff\varphi)^{\sharp_0}\right)(\diff\psi)\right)^{\sharp_0}\right]=\\
=&\psi_{\omega_0}\left(\varphi(\alpha)\right)-\bdiff\Big[\Big(
\left(\psi_{\omega_0}\alpha\right)(\diff\varphi)-\left(\alpha_{\omega_0}(\diff\varphi)^{\sharp_0}\right)\lrcorner\diff\bdiff\psi+\\
&\hspace{3cm}+
\left(\varphi_{\omega_0}\alpha\right)(\diff\psi)
-\bdiff\left(\alpha_{\omega_0}(\diff\varphi)^{\sharp_0}\right)(\diff\psi)
\Big)^{\sharp_0}\Big].
\end{split}
\end{equation}
It is clear that the terms~$\psi_{\omega_0}\left(\varphi(\alpha)\right)$ and~$\left(\psi_{\omega_0}\alpha\right)(\diff\varphi)+\left(\varphi_{\omega_0}\alpha\right)(\diff\psi)$ in \eqref{eq:second_cov} are symmetric in~$\varphi$ and~$\psi$. The remaining terms can be rewritten as
\begin{equation}
\left(\alpha_{\omega_0}(\diff\varphi)^{\sharp_0}\right)\lrcorner\diff\bdiff\psi+\bdiff\left(\alpha_{\omega_0}(\diff\varphi)^{\sharp_0}\right)(\diff\psi)=\bdiff\left[g_0\left(\alpha_{\omega_0}(\diff\varphi),\diff\psi\right)\right]
\end{equation}
which is symmetric in~$\varphi$ and~$\psi$, as~$\alpha$ is compatible with~$\omega_0$.
\end{proof}
The most important consequence of Proposition~\ref{prop:flat} is that the holonomy of~$D$ is trivial, as the K\"ahler class is simply connected. The horizontal transport of~$\alpha$ along a path in~$\Omega$ depends only on the endpoints of the path.
\begin{cor}\label{cor:flat}
For every closed curve~$\gamma_t$ in~$\Omega$, any horizontal lift of~$\gamma_t$ to~$\m{E}$ is again a closed curve.
\end{cor}

\subsection{Formal complexification of the action}\label{sec:complexification}

In this section, we denote by~$\omega_0$ the background symplectic form on~$M$, and we study the action of the group~$\m{G}$ of exact symplectomorphisms of~$(M,\omega_0)$ on the cotangent space of complex structures compatible with~$\omega_0$. The infinitesimal action on~$\cotJ=\cotJ(\omega_0)$ of a function ~$h\in\mrm{Lie}(\m{G})\cong\m{C}^{\infty}(M,\bb{R})$ on~$\cotJ$ is
\begin{equation}
\hat{h}_{J,\alpha}=\left(\m{L}_{X_h}J,\m{L}_{X_h}\alpha\right)
\end{equation}
and since~$\cotJ$ has a complex structure (see~\eqref{eq:app_AC+complexstructure}), we can infinitesimally complexify the action of~$\m{G}$ by setting, for~$h\in\mrm{Lie}(\m{G})^{c}=\m{C}^\infty(M,\bb{C})$,
\begin{equation}
\widehat{h}_{J,\alpha}\coloneqq \widehat{\Re\,(h)}_{J,\alpha}+\bm{I}_{J,\alpha}\widehat{\mrm{Im}(h)}_{J,\alpha}.
\end{equation}
Consider the distribution on~$\cotJ$
\begin{equation}\label{eq:distribuzione}
\begin{split}
\mscr{D}_{J,\alpha}=&\set*{\hat{h}_{J,\alpha}\tc h\in\m{C}^\infty(M,\bb{R})}\oplus\set*{\widehat{\I h}_{J,\alpha}\tc h\in\m{C}^\infty(M,\bb{R})}=\\
=&\set*{\left(\m{L}_{X^{\omega_0}_h}J,\m{L}_{X^{\omega_0}_h}\alpha\right)\tc h\in\m{C}^\infty(M,\bb{R})}\oplus\\
&\oplus\set*{\left(J\m{L}_{X^{\omega_0}_h}J,(\m{L}_{X^{\omega_0}_h}\alpha)J^\transpose+(\m{L}_{X^{\omega_0}_h}J^\transpose)\alpha\right)\tc h\in\m{C}^\infty(M,\bb{R})}.
\end{split}
\end{equation}
We would like to prove that this distribution is integrable, so that the integral leaves of~$\mscr{D}_{J,\alpha}$ can be considered as complexified orbits of the action~$\m{G}\curvearrowright\cotJ$.

We fix an element~$(J,\alpha)$ of~$\cotJ$ such that~$J$ is an integrable complex structure and~$\alpha$ defines an integrable first-order deformation of~$J$. We will show that the K\"ahler class of~$\omega_0$ can be used to parametrize the complex orbit of~$(J,\alpha)$. Denote by~$\Omega$ the K\"ahler class, and consider the set
\begin{equation}
\m{K}(\omega_0)=\set*{(\omega,\psi)\in\Omega\times\mrm{Diff}(M)\tc\psi^*\omega=\omega_0},
\end{equation}
together with the projection on~$\Omega$, which makes~$\m{K}(\omega_0)$ into a~$\mrm{Sympl}(M,\omega_0)$-principal bundle over~$\Omega$. Assume that~$H^1(M)=0$, so that~$\m{K}(\omega_0)$ is a~$\m{G}$-bundle; this is not strictly necessary for our result, as the arguments can be generalized as in~\cite{Donaldson_SymmKahlerHam} at least in the case when~$\Omega=\mrm{c}_1(L)$ for an ample line bundle~$L\to M$.
\begin{prop}\label{prop:complessificazione_cotJ}
Extend~$\alpha$ to a horizontal section of~$\m{E}\to\Omega$. The map~$\Phi_{J,\alpha}:\m{K}(\omega_0)\to\cotJ$ defined by
\begin{equation}
(\omega,\psi)\mapsto\left(\psi^*J,\psi^*\left(\alpha(\omega)\right)\right)
\end{equation}
describes an integral leaf of the distribution~\eqref{eq:distribuzione}.
\end{prop}
Proposition~\ref{prop:complessificazione_cotJ} is inspired by the formal complexification of the Hamiltonian action on~$\m{J}$ described in~\cite[p.$17$]{Donaldson_SymmKahlerHam}. The main difference is that, while~$J$ is compatible with any~$\omega\in[\omega_0]$, the first-order deformation~$\alpha$ is not necessarily compatible with K\"ahler forms other than~$\omega$. This is the reason for extending~$\alpha$ to a horizontal section of~$\m{E}$.

\begin{proof}[Proof of Proposition~\ref{prop:complessificazione_cotJ}]
The compatibility between~$\omega$ and~$\alpha(\omega)$ guarantees that the image of~$\Phi_{J,\alpha}$ is contained in~$\cotJ$. We first prove that the image of~$\Phi_{J,\alpha}$ contains an integral leaf of the ``complex part'' of the distribution, namely
\begin{equation}
\set*{\left(J\m{L}_{X^{\omega_0}_h}J,(\m{L}_{X^{\omega_0}_h}\alpha)J^\transpose+(\m{L}_{X^{\omega_0}_h}J^\transpose)\alpha\right)\tc h\in\m{C}^\infty(M,\bb{R})}.
\end{equation}
For a K\"ahler potential~$h$, consider~$\omega_t=\omega_0+\I\diff\bdiff\,th$ and the vector field~$Y_t=\frac{1}{2}JX^{\omega_t}_h$, where~$X^{\omega_t}_h$ is the Hamiltonian vector field associated to~$h$ with respect to~$\omega_t$. If~$\psi_t$ is the isotopy of the time-dependent vector field~$Y_t$, then~$(\omega_t,\psi_t)\in\m{K}(\omega)$ and~$\diff_t\psi_t^*J$ lies in the ``complex part'' of the distribution~$\mscr{D}$ on~$\m{J}$. We claim that, for~$\alpha_t\coloneqq \alpha(\omega_t)$,
\begin{equation}
t\mapsto\left(\psi_t^*J,\psi_t^*\alpha_t\right)
\end{equation}
is a curve tangent to the complex part of the distribution~$\mscr{D}$. This condition can be rephrased as
\begin{equation}\label{eq:condizioni_complessificazione}
\diff_t\psi_t^*\alpha_t=\frac{1}{2}\psi_t^*\left((\m{L}_{X^{\omega_t}_h}\alpha_t)J^\transpose+(\m{L}_{X^{\omega_t}_h}J^\transpose)\alpha_t\right).
\end{equation}
The derivative of~$\psi_t^*\alpha_t$ can be computed from~\cite[Proposition~$6.4$]{DaSilva_lectures_symp}
\begin{equation}
\diff_t\psi_t^*\alpha_t=\psi_t^*\left(\frac{1}{2}\m{L}_{JX^{\omega_t}_h}\alpha_t+\alpha'_t\right).
\end{equation}
Notice now that, for any vector field~$X$ and any integrable first-order deformation~$\beta$ we have
\begin{equation}
\m{L}_{JX}\beta+\m{L}_{\beta(X)}J^\transpose=(\m{L}_X\beta)J^\transpose+(\m{L}_XJ^\transpose)\beta.
\end{equation}
Then, we can rewrite~$\diff_t\psi^*_t\alpha_t$ as
\begin{equation}
\begin{split}
\diff_t\psi_t^*\alpha_t=\psi_t^*\left(\frac{1}{2}\left(\m{L}_{X^{\omega_t}_h}\alpha_t\right)J^\transpose+\frac{1}{2}\left(\m{L}_{X^{\omega_t}_h}J^\transpose\right)\alpha_t-\frac{1}{2}\m{L}_{\alpha_t(X^{\omega_t}_h)}J^\transpose
+\alpha'_t\right).
\end{split}
\end{equation}
As~$\alpha_t$ is the horizontal lift to~$\m{E}$ of~$\omega_t$
\begin{equation}
\alpha'_t-\frac{1}{2}\m{L}_{\alpha_t(X^{\omega_t}_h)}J^\transpose=\alpha'_t-\I\bdiff\left(\alpha_t(X^{\omega_t}_h)\right)^{1,0}=\alpha'_t-\bdiff\left(\alpha_t(\nabla^{0,1}_th)\right)=0
\end{equation}
using the compatibility of~$\alpha_t$ with~$\omega_t$. Summing up, we have seen that
\begin{equation}
\diff_t\psi_t^*\alpha_t=\frac{1}{2}\psi_t^*\left(\left(\m{L}_{X^{\omega_t}_h}\alpha_t\right)J^\transpose+\left(\m{L}_{X^{\omega_t}_h}J^\transpose\right)\alpha_t\right)
\end{equation}
that is precisely~\eqref{eq:condizioni_complessificazione}.

To prove that the image of~$\Phi_{J,\alpha}$ integrates the whole distribution, it is sufficient to notice that, if~$(\omega,\psi)\in\m{K}(\omega)$ for~$\omega=\omega_0+\I\diff\bdiff h$, then~$\psi$ is the composition of the isotopy of~$Y=\frac{1}{2}JX^{\omega}_h$ with a symplectomorphism of~$\omega$. The assumption~$H^1(M)=0$ implies that this symplectomorphism is an element of~$\m{G}$, so that~$\psi$ will be the composition of the isotopy of~$Y$ with the flow of a Hamiltonian vector field.
\end{proof}
\begin{rmk}
Proposition~\ref{prop:complessificazione_cotJ} explains why we should expect the connection~$D$ of~\eqref{eq:connection_bundle} to be flat, and in particular that~$D$-horizontal lifts only depend on the endpoints of the path, as in Corollary~\ref{cor:flat}. The~$D$-horizontal lifts of paths in~$\Omega$ describe the orbits of the (non-existent) complexification of~$\m{G}$, and the flatness of~$D$ accounts for the fact that these orbits should describe a group action: for any~$(J,\alpha)\in\cotJ$ and any two elements~$\psi_1$,~$\psi_2$ of the mythical group~$\m{G}^{\bb{C}}$, the iterated action~$\psi_1.\left(\psi_2.(J,\alpha)\right)$ should depend only on the composition~$\psi_1\circ\psi_2$, rather than~$\psi_1$ and~$\psi_2$.
\end{rmk}

\subsection{Formal complexification of the equations}\label{sec:complexification_eqs}

Following the classical case of the cscK equation, the formal complexification of the action~$\m{G}\curvearrowright\cotJ$ makes it natural to regard system~\eqref{eq:system_mm_implicit}, i.e.
\begin{equation}
\begin{cases}
\f{m}_{\bm{\Theta}}(J,\alpha)=0\\
\f{m}_{\bm{\Omega}}(J,\alpha)=0
\end{cases}
\end{equation}
as equations for a form~$\omega$, to be found in some prescribed set, keeping instead the complex structure~$J$ fixed.

From the discussion of the cscK problem in~\cite{Donaldson_scalar} we know that to achieve this we should consider the moment maps for~$(J,\alpha)$ in a ``formally complexified orbit'' of the action, i.e. in the image of one of the maps defined in Proposition~\ref{prop:complessificazione_cotJ}. So, fix a reference K\"ahler form~$\omega_0\in\Omega$, and assume that~$(J,\alpha)\in T^*\m{J}(\omega_0)$. We extend~$\alpha$ to a horizontal section of~$\m{E}\to\Omega$; for~$(\omega,\psi)\in\m{K}(\omega_0)$, we can write the \emph{complexified} moment map equations for~$\f{m}_{\bm{\Omega}}$ and~$\f{m}_{\bm{\Theta}}$, as
\begin{equation}\label{eq:complessificate_1}
\begin{cases}
\f{m}_{\bm{\Theta}}\left(\psi^*J,\psi^*\left(\alpha(\omega)\right)\right)=0;\\
\f{m}_{\bm{\Omega}}\left(\psi^*J,\psi^*\left(\alpha(\omega)\right)\right)=0.
\end{cases}
\end{equation}
Focusing on the~$\bm{\Omega}$-equation, recall that the moment map is computed using the metric~$g(J,\omega_0)$ defined from the background symplectic form and the complex structure. Making the dependence on the symplectic form explicit we find
\begin{equation}
\begin{gathered}
\f{m}_{\bm{\Omega}}\left(\psi^*J,\psi^*\left(\alpha(\omega)\right)\right)=
\f{m}_{\bm{\Omega}}\left(g\left(\psi^*J,\omega_0\right),\psi^*\alpha(\omega)\right)=\\
=\psi^*\left(\f{m}_{\bm{\Omega}}\left(g\left(J,\psi^{-1}{}^*\omega_0\right),\alpha(\omega)\right)\right)=\psi^*\left(\f{m}_{\bm{\Omega}}\left(g\left(J,\omega\right),\alpha(\omega)\right)\right).
\end{gathered}
\end{equation}
So, the equation~$\f{m}_{\bm{\Omega}}\left(\psi^*J,\psi^*\alpha(\omega)\right)=0$ is equivalent to
\begin{equation}
\f{m}_{\bm{\Omega}}\left(g\left(J,\omega\right),\alpha(\omega)\right)=0
\end{equation}
hence looking for a solution to the moment map equations along the complexified orbit of~$(J,\alpha)$ is equivalent to keeping~$J$ fixed, moving~$\omega_0$ in its K\"ahler class to some~$\omega\in[\omega_0]$ and simultaneously moving~$\alpha$ along a horizontal lift of the path~$\omega_0\to\omega$. This horizontal lift is independent of the chosen path, by Corollary~\ref{cor:flat}.

The same considerations can be made also for~$\f{m}_{\bm{\Theta}}$, so the complexification of system~\eqref{eq:system_mm_implicit} is described as follows: for a K\"ahler class~$\Omega$ and an infinitesimal deformation class~$\Xi\in H^1(T^{1,0}M)$, fix a reference K\"ahler form~$\omega_0\in\Omega$ and an infinitesimal deformation~$\alpha_0\in\Xi$ that is compatible with~$\omega_0$. Extend~$\alpha_0$ to a horizontal section~$\alpha$ of~$\m{E}\to\Omega$, and look for~$\omega\in\Omega$ such that~$\omega$ and~$\alpha=\alpha(\omega)$ satisfy
\begin{equation}\label{eq:complessificate_definitive}
\begin{dcases}
\m{D}_\omega^*\alpha(\omega)=0,&\\
s(\omega)-\hat{s}+\mrm{div}\,\Re\left(\I Jg_J(\nabla^\sharp\bar{\alpha},\alpha f'(\bar{\alpha}\alpha))+2\nabla^*(\bar{\alpha}\alpha f'(\bar{\alpha}\alpha))\right)=0.
\end{dcases}
\end{equation}
From now on, we will call the horizontal sections of~$\m{E}\to\Omega$ \emph{Higgs fields}, somewhat improperly. By Corollary~\ref{cor:flat}, the set of Higgs fields can be identified with any fibre of~$\m{E}$.

\section{Features of the complexified equations}

\subsection{A variational characterization}\label{sec:energy}

In this section, we focus on the real equation of the complexified system~\eqref{eq:complessificate_intro}. Our goal is to find a variational characterization of the problem, taking advantage of the moment map description. The approach is inspired by analogous considerations in~\cite{GarciaFernandez_PHD}.

For a Higgs field~$\alpha$, consider the~$1$-form on~$\Omega$
\begin{equation}
\left(\sigma_{\alpha}\right)_\omega(\varphi)=-\left\langle\varphi,\f{m}_{\bm{\Omega}}(\omega,J,\alpha(\omega))\right\rangle_{L^2(\omega)}.
\end{equation}
The minus sign in the definition of~$\sigma$ is chosen so that the corresponding energy functional (i.e. a primitive of~$\sigma$) will be convex along the complex directions of the orbit. In a finite-dimensional setting, this is quite easy to check - most proofs of the Kempf-Ness Theorem start from this fact (see for example~\cite[\S$5.4$]{Szekelyhidi_libro}). In our infinite-dimensional situation, where a complexification of the group does not exists, the question of convexity is less clear. See however Section~\ref{sec:toric} for a discussion of the toric case, in which we are able to show that the energy functional is convex along (toric) geodesics.

\begin{prop}\label{prop:sigma_exact}
The form~$\sigma_{\alpha}$ is exact.
\end{prop}
\begin{proof}
As~$\Omega$ is simply connected, we need only prove that~$\sigma_{\alpha}$ is closed. So, we should check that~$\left(\mrm{d}\sigma_{\alpha}\right)_\omega\left(\varphi,\psi\right)=0$ for any~$\omega\in\Omega$ and~$\varphi,\psi\in\m{C}^\infty(M)$. As~$\varphi$ and~$\psi$ are constant vector fields, this boils down to showing that~$\varphi_\omega\left(\sigma_{\alpha}(\psi)\right)$ is symmetric in~$\varphi$ and~$\psi$.

So, consider the derivative
\begin{equation}
\varphi_\omega\left(\sigma_{\alpha}(\psi)\right)=\diff_t\Bigr|_{t=0}\left(\sigma_{\alpha}\right)_{\omega+\I\diff\bdiff t\varphi}(\psi);
\end{equation}
to compute it, we perform a change of variables. Let~$\omega_t=\omega+\I\diff\bdiff t\varphi$, and consider the vector field~$\frac{1}{2}JX^{\omega_t}_\varphi$ and its isotopy~$f_t$. Notice that~$\alpha(\omega_t)$ is the parallel transport of~$\alpha(\omega)$ along~$\omega_t$; as~$f_t^*\omega_t=\omega$,~$\omega_t=f_t^{-1}{}^*\omega$ and we can rewrite our~$1$-form as
\begin{equation}
\begin{split}
\left(\sigma_{\alpha}\right)_{\omega_t}(\psi)=&\left\langle\psi,\f{m}_{\bm{\Omega}}(f_t^{-1}{}^*\omega,J,\alpha(\omega_t))\right\rangle_{L^2(\omega_t)}=\\
=&\left\langle\psi,f_t^{-1}{}^*\f{m}_{\bm{\Omega}}(\omega,f_t^*J,f_t^*\alpha(\omega_t))\right\rangle_{L^2(\omega_t)}=\\
=&\left\langle\psi\circ f_t,\f{m}_{\bm{\Omega}}(\omega,f_t^*J,f_t^*\alpha(\omega_t))\right\rangle_{L^2(\omega)}.
\end{split}
\end{equation}
Taking the derivative of~$\left(\sigma_{\alpha}\right)_{\omega_t}(\psi)$ at~$t=0$ we get
\begin{equation}\label{eq:der_sigma}
\begin{split}
\diff_t\Bigr|_{t=0}\left(\sigma_{\alpha}\right)_{\omega_t}(\psi)=&\left\langle\diff_t\Bigr|_{t=0}\left(\psi\circ f_t\right),\f{m}_{\bm{\Omega}}(\omega,J,\alpha(\omega))\right\rangle_{L^2(\omega)}+\\
&+\left\langle\psi,\diff_t\Bigr|_{t=0}\f{m}_{\bm{\Omega}}(\omega,f_t^*J,f_t^*\alpha(\omega_t))\right\rangle_{L^2(\omega)}.
\end{split}
\end{equation}
For the first term in~\eqref{eq:der_sigma} we directly compute
\begin{equation}
\diff_t\Bigr|_{t=0}\left(\psi\circ f_t\right)=\frac{1}{2}JX^{\omega}_\varphi(\psi)=-\frac{1}{2}g_\omega(\mrm{d}\varphi,\mrm{d}\psi)
\end{equation}
and this is symmetric in~$\varphi$ and~$\psi$. As for the second term, from the definition of a moment map we have
\begin{equation}
\left\langle\psi,\diff_t\Bigr|_{t=0}\f{m}_{\bm{\Omega}}(\omega,f_t^*J,f_t^*\alpha(\omega_t))\right\rangle_{L^2(\omega)}=-\bm{\Omega}\left(\hat{\psi}_{J,\alpha(\omega)},\diff_t\Bigr|_{t=0}f_t^*\left(J,\alpha(\omega_t)\right)\right)
\end{equation}
and~$f_t^*\left(J,\alpha(\omega_t)\right)$ is tangent to the ``purely imaginary'' part of the distribution~$\mscr{D}$ (see the proof of Proposition~\ref{prop:complessificazione_cotJ}), since~$\alpha(\omega_t)$ is a horizontal path in~$\m{E}$. So,
\begin{equation}
\bm{\Omega}\left(\hat{\psi}_{J,\alpha(\omega)},\diff_t\Bigr|_{t=0}f_t^*\left(J,\alpha(\omega_t)\right)\right)=\bm{\Omega}\left(\hat{\psi}_{J,\alpha(\omega)},\bm{I}\hat{\varphi}_{J,\alpha(\omega)}\right)=\bm{G}\left(\hat{\psi}_{J,\alpha(\omega)},\hat{\varphi}_{J,\alpha(\omega)}\right)
\end{equation}
is symmetric in~$\varphi$ and~$\psi$.
\end{proof}

Knowing that~$\sigma_{\alpha}$ is exact, one would hope to find a primitive for it, i.e. a function~$\m{HK}:\Omega\to\bb{R}$ such that
\begin{equation}\label{eq:HK_def}
\mrm{d}_\omega\m{HK}(\varphi)=-\left\langle\varphi,\f{m}_{\bm{\Omega}}(\omega,J,\alpha(\omega))\right\rangle_{L^2(\omega)}.
\end{equation}
It is convenient to write~$\f{m}_{\bm{\Omega}}(\omega,J,\alpha(\omega))$ as the sum of two pieces, from~\eqref{eq:mm_real_explicit}
\begin{equation}
\begin{split}
\f{m}_{\bm{\Omega}}(\omega,&J,\alpha(\omega))=\\
=&\underbrace{s(\omega,J)-\hat{s}}_{\mu(\omega)}+\underbrace{\Re\,\mrm{div}\left(\I J\,\mrm{Tr}\left(\nabla^\sharp\bar{\alpha}\,\alpha f'(\bar{\alpha}\alpha)\right)+2\nabla^*(\bar{\alpha}\alpha f'(\bar{\alpha}\alpha))\right)}_{\f{m}(\omega,\alpha(\omega))}
\end{split}
\end{equation}
correspondingly,~$\sigma$ is decomposed as a scalar curvature term plus a divergence term:
\begin{equation}
\left(\sigma_{\alpha}\right)_\omega(\varphi)=-\left\langle\varphi,\mu(\omega)\right\rangle-\left\langle\varphi,\f{m}(\omega,\alpha(\omega))\right\rangle.
\end{equation}
It is well-known that the K-energy is a primitive for the scalar curvature term in~\eqref{eq:HK_def}. We will show that a primitive for the divergence term is the functional
\begin{equation}
\m{H}_{\alpha}(\omega)=\frac{1}{2}\int f(\bar{\alpha}(\omega)\alpha(\omega))\frac{\omega^n}{n!}
\end{equation}
where~$f$ is the spectral function defining~$\bm{\Omega}$.

\begin{thm}\label{thm:variational}
For any Higgs field~$\alpha$, the real moment map equation in system~\eqref{eq:complessificate_intro} is the Euler-Lagrange equation of
\begin{equation}
\m{HK}_{\alpha}(\omega)\coloneqq \m{M}(\omega)+\m{H}_{\alpha}(\omega)
\end{equation}
where~$\m{M}$ is the K-energy functional. 
\end{thm}
\begin{proof}
Fix a constant vector field~$\varphi\in\Omega$; to compute the derivative of~$\m{H}_{\alpha}$ along~$\varphi$ we proceed as in the proof of Proposition~\ref{prop:sigma_exact}. Let~$\omega_t=\omega+\I\diff\bdiff t\varphi$ be an integral path for~$\varphi$ and let~$\psi_t$ be the isotopy of~$\frac{1}{2}JX^{\omega_t}_\varphi$, so that~$\psi_t^*\omega_t=\omega$. We can write~$\m{H}_{\alpha}$ along the path as
\begin{equation}
\m{H}_{\alpha}(\omega_t)=\frac{1}{2}\int f(\bar{\alpha}(\omega_t)\alpha(\omega_t))\frac{\omega_t^n}{n!}=\frac{1}{2}\int f\left(\psi_t^*\bar{\alpha}(\omega_t)\psi_t^*\alpha(\omega_t)\right)\frac{\omega^n}{n!}.
\end{equation}
To compute the differential of~$\m{H}$, we use again that~$\psi_t^*\left(J,\alpha(\omega_t)\right)$ lies in a purely complex orbit of~$\m{G}\curvearrowright\cotJ$, so that
\begin{equation}
\begin{split}
\varphi_\omega(\m{H}_{\alpha,\omega_0})=&\frac{1}{2}\diff_t\Bigr|_{t=0}\int f\left(\psi_t^*\bar{\alpha}(\omega_t)\psi_t^*\alpha(\omega_t)\right)\frac{\omega^n}{n!}=\frac{1}{2}\int\mrm{d}f(\bm{I}_{J,\alpha(\omega)}\hat{\varphi}_{J,\alpha(\omega)})\frac{\omega^n}{n!}.
\end{split}
\end{equation}
The real moment map is given by
\begin{equation}\label{eq:moment_map_BiquardGauduchon}
\left\langle\f{m}(J,\alpha),h\right\rangle=\frac{1}{2}\int_X\mrm{d}^cf(\hat{h}_{J,\alpha})\frac{\omega^n}{n!}.
\end{equation}
As~$\mrm{d}^cf=-\mrm{d}f\left(J\cdot\right)$, we obtain
\begin{equation}
\varphi_\omega(\m{H}_{\alpha,\omega_0})=\frac{1}{2}\int\mrm{d}f(\bm{I}_{J,\alpha(\omega)}\hat{\varphi}_{J,\alpha(\omega)})\frac{\omega^n}{n!}=-\left\langle\f{m}(\omega,\alpha(\omega)),\varphi\right\rangle.
\end{equation}
\end{proof}

\subsection{Solving the complex moment map equation}\label{sec:complex_eq}

The symplectic form~$\bm{\Theta}$ is of type~$(2,0)$ with respect to the complex structure~$\bm{I}$ of~$\cotJ$. We can use this fact to describe solutions of the equation~$\f{m}_{\bm{\Theta}}(\omega,J,\alpha)=0$ along a complexified orbit of~$\m{G}\curvearrowright\cotJ$.
\begin{proof}[Proof of Proposition~\ref{prop:intro_sol_complex_mm}]
Fix~$(J,\alpha)\in\cotJ$,~$h\in\m{C}^\infty(M,\bb{R})$, and consider the complexified orbit of~$(J,\alpha)$ in the direction~$\I h$ as in Proposition~\ref{prop:complessificazione_cotJ},
\begin{equation}
t\mapsto\psi_t^*(J,\alpha(\omega_t))
\end{equation}
where~$\omega_t=\omega+\I\diff\bdiff\,th$,~$\psi_t$ is the isotopy of~$\frac{1}{2}JX^{\omega_t}_h$ and~$\alpha(\omega_t)$ is the parallel transport of~$\alpha$ along the path~$\omega_t$. Let~$(J_t,\alpha_t)\coloneqq \psi_t^*(J,\alpha(\omega_t))$. For any function~$\varphi\in\m{C}^\infty(M,\bb{R})$, consider the pairing
\begin{equation}
\left\langle\f{m}_{\bm{\Theta}}(\omega_0,J_t,\alpha_t),\varphi\right\rangle_{L^2(\omega_0)}.
\end{equation}
Since~$\f{m}_{\bm{\Theta}}$ is a moment map, the derivative of this is
\begin{equation}
\diff_t\left\langle\f{m}_{\bm{\Theta}}(\omega_0,J_t,\alpha_t),\varphi\right\rangle_{L^2(\omega_0)}=-\bm{\Theta}\left(\hat{\varphi}_{J_t,\alpha_t},\diff_t\left(J_t,\alpha_t\right)\right).
\end{equation}
Recall now that~$\diff_t\left(J_t,\alpha_t\right)=\bm{I}_{J_t,\alpha_t}\hat{h}_{J_t,\alpha_t}$; as~$\bm{\Theta}$ is of type~$(2,0)$,
\begin{equation}
\bm{\Theta}\left(\hat{\varphi}_{J_t,\alpha_t},\diff_t\left(J_t,\alpha_t\right)\right)=\I\bm{\Theta}\left(\hat{\varphi}_{J_t,\alpha_t},\hat{h}_{J_t,\alpha_t}\right).
\end{equation}
Let~$\Phi_s$ be the flow of~$X^{\omega_0}_{h}$. It preserves the form~$\omega_0$, and knowing that~$\f{m}_{\bm{\Theta}}$ is a moment map we find
\begin{equation}
\begin{gathered}
\bm{\Theta}\left(\hat{\varphi}_{J_t,\alpha_t},\hat{h}_{J_t,\alpha_t}\right)=-\diff_s\Bigr|_{s=0}\left\langle\f{m}_{\bm{\Theta}}(\omega_0,\Phi_s^*(J_t,\alpha_t)),\varphi\right\rangle=\\
=-\diff_s\Bigr|_{s=0}\left\langle\f{m}_{\bm{\Theta}}(\omega_0,J_t,\alpha_t),{\Phi_s^{-1}}^*\varphi\right\rangle=\left\langle\f{m}_{\bm{\Theta}}(\omega_0,J_t,\alpha_t),\m{L}_{X^{\omega_0}_h}\varphi\right\rangle=\\
=-\left\langle\mrm{div}\left(\f{m}_{\bm{\Theta}}(\omega_0,J_t,\alpha_t)X^{\omega_0}_h\right),\varphi\right\rangle.
\end{gathered}
\end{equation}
As~$X^{\omega_0}_h$ is divergence-free, we find a differential equation for~$\f{m}_{\bm{\Theta}}(\omega_0,J_t,\alpha_t)$:
\begin{equation}
\diff_t\,\f{m}_{\bm{\Theta}}(\omega_0,J_t,\alpha_t)=\I\,X^{\omega_0}_h\left(\f{m}_{\bm{\Theta}}(\omega_0,J_t,\alpha_t)\right)
\end{equation}
and we are assuming that~$(J_0,\alpha_0)=(J,\alpha)$ is a solution of the moment map equation, i.e.~$\f{m}_{\bm{\Theta}}(\omega_0,J_0,\alpha_0)=0$. Then, the only solution of this PDE is~$\f{m}_{\bm{\Theta}}(\omega_0,J_t,\alpha_t)=0$.
\end{proof}
As a corollary, if we have a solution to the complex equation in~\eqref{eq:complessificate_definitive}, i.e.~$(\alpha_0,\omega_0)\in\m{E}$ such that~$\m{D}_{\omega_0}^*\alpha_0=0$, then also~$(\alpha(\omega),\omega)$ will be a solution of the complex equation for any other~$\omega\in\Omega$, if we extend~$\alpha_0$ to a Higgs field~$\alpha$. This means that once we fix a solution~$(\alpha_0,\omega_0)$ of the complex moment map equation, we can focus on solving the real moment map equation by moving~$\omega_0$ in the K\"ahler class and moving~$\alpha_0$ along the horizontal transport defined by~$D$ on~$\m{E}$.

The general theory of the Lichnerowicz operator on K\"ahler manifolds allows a simple description of the set of solutions of the complex moment map equation, in a fixed fibre of~$\m{E}\to\Omega$. Recall that~$\f{h}$ is the Lie algebra of holomorphic vector fields on~$M$, while~$\f{h}_0$ is the algebra of holomorphic vector fields that admit a holomorphy potential.
\begin{lemma}\label{lem:dim_complex_mm_sols}
Fix~$\omega\in\Omega$, and let~$\eta\in H^1(T^{1,0}M)$ be a first-order deformation class such that~$\eta\cap\m{E}_\omega\not=\emptyset$. Then the set of solutions~$\alpha\in\eta\cap\m{E}_\omega$ of the equation~$\m{D}_\omega^*\alpha=0$ is a space of (complex) dimension~$h^1(M)-\dim\f{h}+\dim\f{h}_0$.
\end{lemma}
\begin{proof}
Notice first that the space of solutions of~$\m{D}^*\alpha=0$ in~$\eta\cap\m{E}_\omega$ is not empty: fix a reference point~$\alpha_0\in\eta\cap\m{E}_\omega$, and consider for~$\varphi\in\m{C}^\infty(M)$ the first-order deformation~$\alpha_\varphi=\alpha_0+\bdiff\left(\nabla^{1,0}\varphi\right)$. This is still an element of~$\eta\cap\m{E}_\omega$, and the equation~$\m{D}^*\alpha_\varphi=0$ is equivalent to
\begin{equation}
\m{D}^*\m{D}\varphi=-\m{D}^*\alpha_0.
\end{equation}
The Fredholm alternative for~$\m{D}^*\m{D}$ implies that this equation has solutions. To compute the dimension of the space of solutions we appeal to \cite[Theorem~$2.5$]{FujikiSchumacher_moduli}. After fixing~$\alpha_1\in\eta\cap\m{E}_\omega$ solving~$\m{D}^*\alpha_1=0$, the space of solutions is in~$1$-to-$1$ correspondence with the space denoted~$\overline{\m{H}}^1_B$ in \cite{FujikiSchumacher_moduli}. Using their notation, the dimension can be computed as
\begin{equation}
\begin{split}
\dim\overline{\m{H}}^1_B=&\dim\hat{\m{H}}_B-\dim\ker\!\bdiff\restriction_{\hat{\m{H}}_B}=\dim`\hat{\m{H}}^1_B+\dim\nabla^{1,0}\m{H}^0_B-\dim\ker\!\bdiff\restriction_{\hat{\m{H}}_B}=\\
=&\dim H^1(M)+\dim\f{h}_0-\dim\f{h}
\end{split}
\end{equation}
since~$H^1(M)\cong`\hat{\m{H}}^1_B$ by \cite[Theorem~$2.5$]{FujikiSchumacher_moduli}.
\end{proof}
For example, in the particular case of a complex curve~$\m{C}$, every first-order deformation class~$\eta$ is contained in~$\m{E}_\omega$ for any~$\omega\in\Omega$, since the compatibility condition is trivial in dimension~$1$. If~$g(\m{C})=0$ then~$\dim H^1(T\m{C})=0$,~$h^1(M)=0$ and~$\dim\f{h}=\dim\f{h}_0$, so the only solution of~$\m{D}^*\alpha=0$ is~$\alpha=0$. If~$g(\m{C})=1$ instead we have~$\dim H^1(T\m{C})=1$,~$h^1(M)=\dim\f{h}$ and~$\dim\f{h}_0=0$, so that there is a one-dimensional space of solutions of~$\m{D}^*\alpha=0$. Finally, for~$g(\m{C})\geq 2$ we have~$\mrm{dim}\,H^1(T\m{C})=3g-3$ and~$\mrm{dim}\,H^1(M)=g$, while there are no holomorphic vector fields. Lemma~\ref{lem:dim_complex_mm_sols} then tells us that for any~$\omega\in\Omega$ the set of solutions to~$\m{D}^*\alpha=0$ is a space of dimension~$4g-3$. This result was already obtained in~\cite{ScarpaStoppa_HcscK_curve}, from a slightly different point of view.

\begin{rmk}\label{rmk:deformation_polarized}
Combining the results of \cite[Theorem~$2.4$]{FujikiSchumacher_moduli} and \cite[Theorem~$2.9$]{FujikiSchumacher_moduli} in the polarized case~$\Omega=\mrm{c}_1(L)$, we see that solutions of the complex moment map equation~$\m{D}^*\alpha=0$ are in~$1$-to-$1$ correspondence with~$H^1(M,\m{E}_L)$, the space of first-order deformations of~$(M,L)$. In other words, for each fixed K\"ahler metric~$\omega$, the complex moment map equation gives a way to choose a unique representative of any class in~$H^1(M,\m{E}_L)$.
\end{rmk}

\subsection{Linearization and a perturbation result}\label{sec:linearization}

In this Section we investigate the linearized equation in a neighbourhood of a solution of equation~\eqref{eq:mm_real_explicit}, proving Theorem~\ref{thm:intro_linearizzazione}.

\begin{proof}[Proof of Theorem~\ref{thm:intro_linearizzazione}]
The linearization of the scalar curvature is a fourth order elliptic operator; for~$\varphi\in\m{C}^\infty(M)$, the highest order term in the linearization of the scalar curvature along~$\omega+\I\diff\bdiff t\varphi$ is
\begin{equation}
\diff_{t=0}s(\omega+\I\diff\bdiff t\varphi)=-g^{a\bar{b}}g^{c\bar{d}}\diff_a\diff_{\bar{b}}\diff_c\diff_{\bar{d}}\varphi+\mbox{l.o.t.}
\end{equation}
To compute the linearization of the divergence term we let~$\alpha_t\coloneqq \alpha(\omega_t)$; notice that
\begin{equation}
\diff_{t=0}\alpha(\omega_t)=\bdiff\left(\alpha(\diff\varphi)^{\sharp}\right)=g^{a\bar{c}}\alpha\indices{_{\bar{c}}^d}\diff_d\diff_{\bar{b}}\varphi\mrm{d}\bar{z}^b\otimes\diff_a+\mrm{l.o.t.}
\end{equation}
is of second order in~$\varphi$, so the linearization of the divergence term, namely
\begin{equation}
\mrm{div}_t\,\Re\left(\I Jg_t\left(\nabla_t^{\sharp_t}\bar{\alpha}_t,\alpha_t\,f'(\bar{\alpha}_t\alpha_t)\right)+2\nabla_t^{*_t}\left(f'(\bar{\alpha}_t\alpha_t)\bar{\alpha}_t\alpha_t\right)\right)
\end{equation}
is a fourth order operator. Considering only the top-order terms in~$\varphi$ we get
\begin{equation}\label{eq:lineariz_firstterm}
\begin{split}
\diff_{t=0}&\left[\I Jg_t\left(\nabla_t^{\sharp_t}\bar{\alpha}_t,\alpha_t\,f'(\bar{\alpha}_t\alpha_t)\right)\right]=\\
=&-g\left(\nabla^{1,0}\diff_{t=0}\bar{\alpha}_t,\alpha\,f'(\bar{\alpha}\alpha)\right)+g\left(\nabla^{0,1}\diff_{t=0}\bar{\alpha}_t,\alpha\,f'(\bar{\alpha}\alpha)\right)=\\
=&-g^{a\bar{b}}\diff_{\bar{b}}\diff_c\diff_{\bar{e}}\varphi\,g^{c\bar{d}}(\alpha f'(\bar{\alpha}\alpha)\bar{\alpha})\indices{_{\bar{d}}^{\bar{e}}}\diff_a+g^{a\bar{b}}\diff_a\diff_c\diff_{\bar{e}}\varphi\,g^{c\bar{d}}(\alpha f'(\bar{\alpha}\alpha)\bar{\alpha})\indices{_{\bar{d}}^{\bar{e}}}\diff_{\bar{b}}.
\end{split}
\end{equation}
As~$g^{-1}(\alpha f'(\bar{\alpha}\alpha)\bar{\alpha})$ is Hermitian,~\eqref{eq:lineariz_firstterm} is purely imaginary. To compute the principal symbol of~$\m{L}$ then, it is enough to consider the derivative
\begin{equation}\label{eq:lineariz_secondpart}
\diff_{t=0}\mrm{div}_t\,\Re\left(2\nabla_t^{*_t}\left(f'(\bar{\alpha}(\omega_t)\alpha(\omega_t))\bar{\alpha}(\omega_t)\alpha(\omega_t)\right)\right).
\end{equation}
We need an expression for the second directional derivative of spectral functions; such an expression can be found for example in~\cite{derivatives_eigenvalues}, but it only holds for spectral functions of Hermitian matrices. In order to use this result, we first perform a change of variables.

Considering the matrices associated to~$\alpha_t$ and~$g_t$ in a local system of complex coordinates, we can write~$\bar{\alpha}_t=q_tg_t^{-1}$ for a complex symmetric matrix~$q_t$, corresponding to a tensor~$q_t=q(t)_{ab}\mrm{d}z^a\odot\mrm{d}z^b\in\mrm{Sym}^2(TM)$. The composition~$\bar{\alpha}_t\alpha_t$ is written as
\begin{equation}
\bar{\alpha}_t\alpha_t=q_tg_t^{-1}\bar{q}_tg_t^{-1}=g_t^{\frac{1}{2}}g_t^{-\frac{1}{2}}q_tg_t^{-1}\bar{q}_tg_t^{-\frac{1}{2}}g_t^{-\frac{1}{2}}
\end{equation}
and the matrix~$H_t\coloneqq g_t^{-\frac{1}{2}}q_tg_t^{-1}\bar{q}_tg_t^{-\frac{1}{2}}$ is Hermitian. The tensor~$f'(\bar{\alpha}_t\alpha_t)\bar{\alpha}_t\alpha_t$ can then be rewritten as
\begin{equation}
f'(\bar{\alpha}_t\alpha_t)\bar{\alpha}_t\alpha_t=g_t^{\frac{1}{2}}f'(H_t)H_tg_t^{-\frac{1}{2}}.
\end{equation}
As we are only interested in the highest-order derivatives of~$\varphi$ in~\eqref{eq:lineariz_secondpart}, we just need to consider
\begin{equation}
\begin{split}
\Re\,\mrm{div}&\left(2\nabla^*\diff_{t=0}\left(g_t^{\frac{1}{2}}f'(H_t)H_tg_t^{-\frac{1}{2}}\right)\right)=\\
&=-2\,\Re\left(g^{a\bar{c}}\nabla_b\nabla_{\bar{c}}\diff_{t=0}\left(g_t^{\frac{1}{2}}f'(H_t)H_tg_t^{-\frac{1}{2}}\right)\indices{_a^b}\right)
\end{split}
\end{equation}
and we can furthermore assume that~$g_{t=0}=\bb{1}$,~$q_{t=0}=\mrm{diag}\left(\sqrt{\lambda_1},\dots,\sqrt{\lambda_n}\right)$. To simplify notation we will also assume that the eigenvalues of~$\alpha_t$ are all distinct, for small enough~$t$. With these hypothesis we obtain
\begin{equation}
\diff_{t=0}\left(g_t^{\frac{1}{2}}f'(H_t)H_tg_t^{-\frac{1}{2}}\right)\indices{_a^b}=\frac{1}{2}\left(\lambda_b\,k'(\lambda_b)-\lambda_a\,k'(\lambda_a)\right)\diff_a\diff_{\bar{b}}\varphi+\diff_{t=0}\left(f'(H_t)H_t\right)^{\bar{a}b}
\end{equation}
here, and in the remainder of the proof, we do not use the summation convention on repeated indices. The derivative of~$H_t$ can be computed from~$\diff_{t=0}q_t$; a straightforward computation using~\eqref{eq:hor_lift} gives
\begin{equation}
\begin{split}
\diff_{t=0}\left(q_t\right)_{cd}=&\sqrt{\lambda_d}\diff_c\diff_{\bar{d}}\varphi+\sqrt{\lambda_c}\diff_d\diff_{\bar{c}}\varphi\\
\diff_{t=0}\left(H_t\right)^{\bar{a}b}=&\frac{\lambda_a+\lambda_b}{2}\diff_a\diff_{\bar{b}}\varphi+\sqrt{\lambda_a\lambda_b}\diff_{\bar{a}}\diff_b\varphi.
\end{split}
\end{equation}
At this point, we can obtain an expression for the derivative of~$f'(H_t)H_t$ from~\cite[Theorem~$3.1$]{derivatives_eigenvalues}.
\begin{equation}
\begin{gathered}
\diff_{t=0}\left(f'(H_t)H_t\right)^{\bar{a}b}=\sum_c\diff_{t=0}\left(f'(H_t)\right)^{\bar{a}c}\delta_{cb}\lambda_b+\delta_{ac}k'(\lambda_a)\diff_{t=0}H_t^{\bar{c}b}=\\
=\left(\lambda_b\,\delta^{ab}\,k''(\lambda_a)
+\lambda_b(1-\delta_{ab})\frac{k'(\lambda_a)-k'(\lambda_b)}{\lambda_a-\lambda_b}
+k'(\lambda_a)
\right)\cdot\\
\hfill\cdot\left(\frac{\lambda_a+\lambda_b}{2}\diff_a\diff_{\bar{b}}\varphi+\sqrt{\lambda_a\lambda_b}\diff_{\bar{a}}\diff_b\varphi\right).
\end{gathered}
\end{equation}
Putting all this together, we obtain
\begin{equation}\label{eq:linearizz_acc1}
\begin{gathered}
\sum_{a,b,c}g^{a\bar{c}}\nabla_b\nabla_{\bar{c}}\diff_{t=0}\left(g_t^{\frac{1}{2}}f'(H_t)H_tg_t^{-\frac{1}{2}}\right)\indices{_a^b}=\hfill\\
=\mrm{l.o.t.}-\sum_{a,b}\left(\frac{1}{2}\left(\lambda_b\,k'(\lambda_b)-\lambda_a\,k'(\lambda_a)\right)\diff_b\diff_{\bar{a}}\diff_a\diff_{\bar{b}}\varphi\right)-\\
-\sum_{a,b}\left(\lambda_b\,\delta^{ab}\,k''(\lambda_a)
+\lambda_b(1-\delta_{ab})\frac{k'(\lambda_a)-k'(\lambda_b)}{\lambda_a-\lambda_b}
+k'(\lambda_a)
\right)\cdot\\
\hfill\cdot\left(\frac{\lambda_a+\lambda_b}{2}\diff_b\diff_{\bar{a}}\diff_a\diff_{\bar{b}}\varphi+\sqrt{\lambda_a\lambda_b}\diff_b\diff_{\bar{a}}\diff_{\bar{a}}\diff_b\varphi\right).
\end{gathered}
\end{equation}
The first sum on the right-hand side of the equation~\eqref{eq:linearizz_acc1} 
is anti-symmetric in the indices~$a,b$, so that sum vanishes. Then, equation~\eqref{eq:linearizz_acc1} shows that the principal symbol of the linearization is
\begin{equation}
\begin{gathered}
\sigma(\xi)=
-\sum_{a,b}\left(\lambda_b\,\delta^{ab}\,k''(\lambda_a)
+\lambda_b(1-\delta_{ab})\frac{k'(\lambda_a)-k'(\lambda_b)}{\lambda_a-\lambda_b}
+k'(\lambda_a)
\right)\cdot\\
\hfill\cdot\left(\frac{\lambda_a+\lambda_b}{2}\abs{\xi^b}^2\abs{\xi^a}^2+\sqrt{\lambda_a\lambda_b}\,\Re\left((\xi^a)^2(\bar{\xi}^b)^2\right)\right).
\end{gathered}
\end{equation}
We can reorganize the sum as
\begin{equation}
\begin{split}
\sigma(\xi)=&-\sum_a\left(\lambda_a\,k''(\lambda_a)+k'(\lambda_a)\right)2\lambda_a\abs{\xi^a}^4+\\
&-\sum_{a\not= b}\left(\lambda_b\frac{k'(\lambda_a)-k'(\lambda_b)}{\lambda_a-\lambda_b}+k'(\lambda_a)\right)\cdot\\
&\phantom{-\sum_{a\not= b}\cdot}\cdot\left(\frac{\lambda_a+\lambda_b}{2}\abs{\xi^b}^2\abs{\xi^a}^2+\sqrt{\lambda_a\lambda_b}\,\Re\left((\xi^a)^2(\bar{\xi}^b)^2\right)\right)
\end{split}
\end{equation}
and under our assumptions on~$k$, it is clear that the first sum on the right-hand side is strictly negative, for~$\xi\not=0$. We will show that the same holds for the second sum. To do so, it will be enough to consider two fixed indices~$a,b$ and show that
\begin{equation}
\begin{split}
&\left(\lambda_b\frac{k'(\lambda_a)-k'(\lambda_b)}{\lambda_a-\lambda_b}+k'(\lambda_a)\right)\!\left(\frac{\lambda_a+\lambda_b}{2}\abs{\xi^b}^2\abs{\xi^a}^2+\sqrt{\lambda_a\lambda_b}\,\Re\left((\xi^a)^2(\bar{\xi}^b)^2\right)\right)+\\
&+\left(\lambda_a\frac{k'(\lambda_b)-k'(\lambda_a)}{\lambda_b-\lambda_a}+k'(\lambda_b)\right)\!\left(\frac{\lambda_b+\lambda_a}{2}\abs{\xi^a}^2\abs{\xi^b}^2+\sqrt{\lambda_b\lambda_a}\,\Re\left((\xi^b)^2(\bar{\xi}^a)^2\right)\right)
\end{split}
\end{equation}
is positive. As~$\Re\left((\xi^b)^2(\bar{\xi}^a)^2\right)>-\abs{\xi^a}^2\abs{\xi^b}^2
$, this is greater than
\begin{equation}
\left((\lambda_a+\lambda_b)\frac{k'(\lambda_a)-k'(\lambda_b)}{\lambda_a-\lambda_b}+k'(\lambda_a)+k'(\lambda_b)\right)\left(\frac{\lambda_a+\lambda_b}{2}-\sqrt{\lambda_a\lambda_b}\right)\abs{\xi^a}^2\abs{\xi^b}^2
\end{equation}
which is clearly positive.

The second part of Theorem~\ref{thm:intro_linearizzazione} follows almost directly from the proof of Proposition~\ref{prop:sigma_exact}, and is in fact a general feature of moment map equations in K\"ahler geometry. Assume that~$\alpha$ is a Higgs field on~$\Omega$, and consider the linearization along the path~$\omega+\I\diff\bdiff t\varphi$ for some~$\varphi\in\m{C}^\infty(M)$. 

As in the proof of Proposition~\ref{prop:sigma_exact}, we let~$\omega_t=\omega+\I\diff\bdiff t\varphi$,~$X_t\coloneqq \frac{1}{2}JX^{\omega_t}_\varphi$, and~$\psi_t$ will be the isotopy of~$X_t$. If we denote by~$\m{L}_\alpha(\varphi)$ the linearization of~$\f{m}_{\bm{\Omega}}$ at~$(\omega,\alpha)$, we have
\begin{equation}\label{eq:linearizz_implicit}
\begin{split}
\m{L}_\alpha(\varphi)=&\diff_{t=0}\f{m}_{\bm{\Omega}}\left(\omega_t,J,\alpha(\omega_t)\right)=\diff_{t=0}\ \psi_t^{-1}{}^*\left[\f{m}_{\bm{\Omega}}\left(\omega,\psi_t^*\left(J,\alpha(\omega_t)\right)\right)\right]=\\
=&\m{L}_{-\frac{1}{2}\nabla\varphi}\left(\f{m}_{\bm{\Omega}}(\omega,J,\alpha)\right)+\diff_{t=0}\f{m}_{\bm{\Omega}}\left(\omega,\psi_t^*\left(J,\alpha(\omega_t)\right)\right).
\end{split}
\end{equation}
To show that~$\m{L}_\alpha$ is self-adjoint (with respect to the~$L^2$-pairing defined by~$\omega$) we should check that
\begin{equation}
\left\langle\m{L}_\alpha(\varphi),u\right\rangle=-\frac{1}{2}\left\langle\m{L}_{\nabla\varphi}\left(\f{m}_{\bm{\Omega}}(\omega,J,\alpha)\right),u\right\rangle
+\left\langle
\diff_{t=0}\f{m}_{\bm{\Omega}}\left(\omega,\psi_t^*\left(J,\alpha(\omega_t)\right)\right),u\right\rangle
\end{equation}
is symmetric in~$\varphi$ and~$u$, for every~$\varphi,u\in\m{C}^\infty(M)$. But this is precisely what we already proved in Proposition~\ref{prop:sigma_exact} to show that~$\sigma_\alpha$ is exact.

We now consider the kernel of the linearized operator around a solution. Assuming that~$\omega$ is a solution of the complexified moment map equation, the first term in~\eqref{eq:linearizz_implicit} vanishes. Then,~$\varphi\in\mrm{ker}(\m{L}_\alpha)$ if and only if, for every~$u\in\m{C}^{\infty}(M)$
\begin{equation}
\left\langle\diff_{t=0}\f{m}_{\bm{\Omega}}\left(\omega,\psi_t^*\left(J,\alpha(\omega_t)\right)\right),u\right\rangle_{L^2(\omega)}=0.
\end{equation}
The calculations in the proof of Proposition~\ref{prop:sigma_exact}, show that this is equivalent to
\begin{equation}
\bm{G}\left(\hat{u}_{J,\alpha},\hat{\varphi}_{J,\alpha}\right)=0.
\end{equation}
It follows that~$\varphi\in\mrm{ker}(\m{L}_\alpha)$ if and only if~$\hat{\varphi}_{J,\alpha}=0$, i.e.~$\m{L}_{X^\omega_\varphi}(J,\alpha)=0$.
\end{proof}
The implicit function theorem allows us to obtain various existence results for small deformations of~$\omega$ and~$\alpha$ around a solution of equation~\eqref{eq:complessificate_definitive}, under some assumptions on the automorphism group of~$(M,J,\alpha)$. We state two of such results, that might be of particular interest.
\begin{cor}
Assume that~$(\omega,\alpha)$ is a solution of~\eqref{eq:complessificate_definitive}, and assume that~$\mrm{ker}\,\m{L}_\alpha$ is trivial. Then
\begin{enumerate}
\item for any other first-order deformation of the complex structure~$\alpha'$ that is close enough to~$\alpha$, there is a K\"ahler form~$\omega'\in[\omega]$ such that~$(\omega',\alpha')$ solve~\eqref{eq:complessificate_definitive}.
\item for any spectral function~$\tilde{f}$ sufficiently close to~$f$, the moment map~\eqref{eq:mm_real_explicit} with respect to~$\tilde{f}$ also has a zero.
\end{enumerate}
\end{cor}

\section{A stability condition}\label{sec:Futaki}

This section aims to use the properties of the moment maps in~\eqref{eq:system_mm_implicit} to define an integral invariant of a subgroup of~$\mrm{Aut}(M,J)$, that vanishes if there is a solution of the real moment map equation~$\f{m}_{\bm{\Omega}}(\omega,\alpha(\omega))=0$. To define this generalization of the Futaki invariant and prove Proposition~\ref{prop:Futaki}, we need some preliminary considerations on the vector bundle of Section~\ref{sec:bundle_class}.

Let~$J$ be a complex structure on~$M$, and let~$\Omega$ be a K\"ahler class. The set of holomorphic vector fields on~$M$ is the Lie algebra of~$\mrm{Aut}(M,J)$. We are mostly interested in the subalgebra~$\f{h}_0\subseteq H^0(T^{1,0}M)$ of holomorphic vector fields that admit a holomorphy potential with respect to some K\"ahler form in~$\Omega$. We denote by~$\mrm{Aut}_0(M,J)$ the corresponding subgroup of~$\mrm{Aut}(M,J)$;~$\mrm{Aut}_0(M,J)$ is commonly called the reduced automorphism group of~$(M,J)$. The algebra~$\f{h}_0$ does not depend on the choice of the reference K\"ahler form (see~\cite{LeBrun-Simanca}), hence the reduced automorphism group only depends on the complex structure of~$M$.

\begin{lemma}\label{lemma:hol_pullbak_connection}
For any section~$\alpha\in\Gamma(\Omega,\m{E})$ and every~$\psi\in\mrm{Aut}_0(M,J)$, the map
\begin{equation}
\psi.\alpha:\omega\mapsto\psi^{-1}{}^*\left[\alpha(\psi^*\omega)\right]
\end{equation}
also defines a section of~$\m{E}$. Moreover,~$\alpha$ is horizontal if and only if~$\psi.\alpha$ is.
\end{lemma}

\begin{proof}
First notice that the pullback by any~$\psi\in\mrm{Aut}_0$ preserves the class~$\Omega$, as~$\psi=\mrm{exp}(X)$ for a vector field~$X$ that has a holomorphy potential. The first claim follows easily from this: if~$\alpha\in\m{A}^{0,1}(T^{1,0}M)$ is compatible with~$\omega\in\Omega$, then~$\psi^*\alpha$ will also be compatible with~$\psi^*\omega$.

To prove the second claim, just notice that the pullback by a holomorphic map preserves the connection~\eqref{eq:connection_bundle} on~$\m{E}$.
\end{proof}

Lemma~\ref{lemma:hol_pullbak_connection} tells us that~$\alpha\mapsto\psi.\alpha$ is a~$\mrm{Aut}_0(M,J)$-action on the space of Higgs fields. We can now define the automorphism group that is most relevant for the study of system~\eqref{eq:system_mm_implicit}.
\begin{definition}
For any Higgs field~$\alpha$, we let
\begin{equation}\label{eq:Aut_alpha_def}
\begin{gathered}
\mrm{Aut}_\alpha(M,J)=\set*{\psi\in\mrm{Aut}_0(M,J)\tc\psi.\alpha=\alpha}\\
\f{h}_\alpha=\mrm{Lie}(\mrm{Aut}_\alpha(M,J))\subseteq\f{h}_0
\end{gathered}
\end{equation}
\end{definition}
We will describe~$\f{h}_\alpha$ more explicitly in Lemma~\ref{lem:h_alpha}. As the complex structure~$J$ on~$M$ is usually fixed, we will write just~$\mrm{Aut}_\alpha$ instead of~$\mrm{Aut}_\alpha(M,J)$. Notice that if~$\alpha=0$, we recover the reduced automorphism group~$\mrm{Aut}_0$ and Lie algebra~$\f{h}_0$.

\begin{rmk}\label{rmk:aut_alpha}
By the uniqueness of horizontal lifts, if two horizontal sections of~$\m{E}$ coincide at any point, they will coincide everywhere. Then, for~$\psi\in\mrm{Aut}_0$ and a Higgs field~$\alpha$, we can check if~$\psi\in\mrm{Aut}_\alpha$ at just one point:~$f.\alpha=\alpha$ if and only if~$\psi.\alpha(\omega)=\alpha(\omega)$ for some~$\omega\in\Omega$.
\end{rmk}

\begin{rmk}
For~$(J,\alpha)\in\cotJ$, consider the stabilizer~$\m{G}_{J,\alpha}$ of~$(J,\alpha)$ in the group of Hamiltonian diffeomorphisms. A diffeomorphism~$\psi\in\m{G}$ is in~$\m{G}_{J,\alpha}$ if and only if~$\psi^*J=J$ and~$\psi^*\alpha=\alpha$. As~$\psi$ is a symplectomorphism,~$\psi^*\alpha=\alpha$ if and only if~$\psi.\alpha=\alpha$ (see the definition of~$\psi.\alpha$ in Lemma~\ref{lemma:hol_pullbak_connection}). Hence~$\m{G}_{J,\alpha}\subseteq\mrm{Aut}_\alpha$, so~$\mrm{Aut}_\alpha$ contains the complexification of~$\m{G}_{J,\alpha}$. When a solution of system~\eqref{eq:system_mm_implicit} exists, we expect~$\mrm{Aut}_\alpha$ to coincide with the complexification of~$\m{G}_{J,\alpha}$; this generalization of Matsushima's criterion should be just a formal consequence of the properties of moment maps.
\end{rmk}

We can characterize the Lie algebra~$\f{h}_\alpha$ in terms of holomorphy potentials.
\begin{lemma}\label{lem:h_alpha}
Fix a Higgs field~$\alpha$ and a holomorphic vector field~$X\in\f{h}_0$. For a K\"ahler form~$\omega$, let~$h(X,\omega)\in\m{C}^\infty(M,\bb{C})$ be the holomorphy potential of~$X$. Then~$X\in\f{h}_\alpha$ if and only if~$\m{L}_{X^{\omega}_{h(X,\omega)}}\alpha(\omega)\in\mrm{End}(T^*_{\bb{C}}M)$ commutes with~$J$, for some (hence any)~$\omega\in\Omega$.
\end{lemma}
\begin{proof}
Fix a K\"ahler form~$\omega$, and let for simplicity~$\alpha=\alpha(\omega)$,~$h=h(X,\omega)$. For any vector field~$Y$, we have
\begin{equation}
\begin{split}
\left(\m{L}_Y\alpha\right)\indices{_b^a}=&\alpha\indices{_{\bar{c}}^a}\diff_bY^{\bar{c}};\\
\left(\m{L}_Y\alpha\right)\indices{_{\bar{b}}^a}=&\diff_{\bar{b}}\left(\alpha\indices{_{\bar{c}}^a}Y^{\bar{c}}\right)+Y^c\diff_c\alpha\indices{_{\bar{b}}^a}-\alpha\indices{_{\bar{b}}^c}\diff_cY^a;\\
\left(\m{L}_Y\alpha\right)\indices{_b^{\bar{a}}}=&0;\\
\left(\m{L}_Y\alpha\right)\indices{_{\bar{b}}^{\bar{a}}}=&-\alpha\indices{_{\bar{b}}^c}\diff_cY^{\bar{a}}.
\end{split}
\end{equation}
Substituting~$X^\omega_h=J\nabla h$ to~$Y$, we see that~$\m{L}_{X^\omega_h}\alpha$ commutes with~$J$ if and only if
\begin{equation}\label{eq:h_alpha}
-\I\diff_{\bar{b}}\left(\alpha\indices{_{\bar{c}}^a}\nabla^{\bar{c}}h\right)+\I\left(\nabla^ch\diff_c\alpha\indices{_{\bar{b}}^a}-\alpha\indices{_{\bar{b}}^c}\diff_c\nabla^ah\right)=0.
\end{equation}
Now, let~$\varphi_t\in\mrm{Aut}_0$ be the exponential of~$X$. By Remark~\ref{rmk:aut_alpha},~$\varphi_t$ is an element of~$\mrm{Aut}_{\alpha}$ if and only if
\begin{equation}
\varphi_t^*\alpha=\alpha(\varphi_t^*\omega)
\end{equation}
and taking the derivative at~$t=0$ we find that~$X\in\f{h}_\alpha$ if and only if
\begin{equation}
\m{L}_X\alpha=\diff_t\Bigr|_{t=0}\alpha(\varphi_t^*\omega)=\bdiff\left(\alpha(\diff h)^\sharp\right).
\end{equation}
In terms of local holomorphic coordinates, this is written as
\begin{equation}
\nabla^ch\diff_c\alpha\indices{_{\bar{b}}^a}-\alpha\indices{_{\bar{b}}^c}\diff_c\nabla^ah=\diff_{\bar{b}}\left(\alpha\indices{_{\bar{c}}^d}\diff_dh g^{a\bar{c}}\right)
\end{equation}
that is precisely~\eqref{eq:h_alpha}.
\end{proof}

By analogy with the cscK case, we define a character of~$\mrm{Aut}_\alpha$ that is the analogue of a Futaki invariant for the equation~$\f{m}_{\bm{\Omega}}=0$.
\begin{definition}
Let~$\Omega$ be a K\"ahler class on~$M$, and let~$\alpha$ be a horizontal section of~$\m{E}\to\Omega$. For~$\omega\in\Omega$ and~$X\in\f{h}_\alpha$, let~$h(\omega,X)$ be the holomorphy potential of~$X$ with respect to~$\omega$. We define the \emph{Futaki invariant of system~\eqref{eq:system_mm_implicit}} as
\begin{equation}
\m{F}_\alpha(X)=\int_Mh(\omega,X)\,\f{m}_{\bm{\Omega}}\left(\omega,J,\alpha(\omega)\right)\frac{\omega^n}{n!}.
\end{equation}
\end{definition}
We prove that~$\m{F}_\alpha$ is an invariant of the K\"ahler class adapting the proof in~\cite{Bourguignon_Futaki} of the analogous result for the cscK equation.
\begin{prop}\label{prop:Futaki_vanish}
For every~$X\in\f{h}_\alpha$,~$\m{F}_\alpha(X)$ does not depend on the choice of~$\omega\in\Omega$. In particular, if there is a solution to~$\f{m}_{\bm{\Omega}}=0$ then~$\m{F}_\alpha\equiv 0$.
\end{prop}

\begin{proof}
Consider the action by pull-backs of~$\mrm{Aut}_0$ on the K\"ahler class~$\Omega$. We claim that the~$1$-form~$\sigma_\alpha$ defined in Section~\ref{sec:energy} is invariant for the induced action of~$\mrm{Aut}_\alpha$.

Assuming this for the moment, the conclusion follows easily from Proposition~\ref{prop:sigma_exact}: the invariance of~$\sigma_\alpha$ implies that~$\m{L}_X\sigma_\alpha=0$ for every~$X\in\f{h}_0$. As~$\sigma_\alpha$ is closed, this implies that~$\sigma_\alpha(\hat{X})$ is constant, where~$\hat{X}$ denotes the infinitesimal action of~$X$ on~$\Omega$. It can be easily checked that this infinitesimal action is precisely~$\hat{X}_\omega=h(X,\omega)$, hence
\begin{equation}
\sigma_\alpha(\hat{X})=-\int_Mh(X,\omega)\,\f{m}_{\bm{\Omega}}(\omega,J,\alpha)\frac{\omega^n}{n!}=-\m{F}_\alpha(X)
\end{equation}
is a constant function of~$\omega$.

It remains to prove the invariance of~$\sigma_\alpha$. Denote by~$k$ the action~$\mrm{Aut}_\alpha\curvearrowright\Omega$, i.e.~$k_\psi(\omega)\coloneqq \psi^*\omega$. We want to show that, for every biholomorphism~$\psi\in\mrm{Aut}_\alpha$, K\"ahler form~$\omega$, and~$\varphi\in\m{C}^\infty(M,\bb{R})$,
\begin{equation}
\left(k_\psi^*\sigma_\alpha\right)_\omega(\varphi)=\left(\sigma_\alpha\right)_\omega(\varphi).
\end{equation}
Recall that we identify vector fields on~$\Omega$ with functions on~$M$, via the~$\diff\bdiff$-Lemma. As every~$\psi\in\mrm{Aut}_\alpha$ is a biholomorphism, the differential of the action of~$\psi$ on~$\Omega$ is
\begin{equation}
\left(\mrm{d}k_\psi\right)_\omega(\varphi)=\psi^*\varphi\in T_{\psi^*\omega}\Omega.
\end{equation}
We can then compute
\begin{equation}
\left(k_\psi^*\sigma_\alpha\right)_\omega(\varphi)=\left(\sigma_\alpha\right)_{\psi^*\omega}(\psi^*\varphi)=-\int_M\psi^*\varphi\,\f{m}_{\bm{\Omega}}(\psi^*\omega,J,\alpha(\psi^*\omega))\frac{\psi^*\omega^n}{n!}.
\end{equation}
The key observation is the same used in the proof of Proposition~\ref{prop:sigma_exact}:
\begin{equation}
\f{m}_{\bm{\Omega}}(\psi^*\omega,J,\alpha(\psi^*\omega))=\psi^*\left(\f{m}_{\bm{\Omega}}(\omega,\psi^{-1}{}^*J,\psi^{-1}{}^*\alpha(\psi^*\omega))\right)=\psi^*\left(\f{m}_{\bm{\Omega}}(\omega,J,\alpha(\omega))\right)
\end{equation}
where the second equality comes from~$\psi\in\mrm{Aut}_\alpha$. Then, we find
\begin{equation}
\left(k_\psi^*\sigma_\alpha\right)_\omega(\varphi)=-\int_M\psi^*\varphi\,\psi^*\left(\f{m}_{\bm{\Omega}}(\omega,J,\alpha(\omega))\right)\frac{\psi^*\omega^n}{n!}=-\int_M\varphi\,\f{m}_{\bm{\Omega}}(\omega,J,\alpha(\omega))\frac{\omega^n}{n!}
\end{equation}
which is exactly~$\left(\sigma_\alpha\right)_\omega(\varphi)$.
\end{proof}

The expression of the Futaki invariant can be greatly simplified, using the decomposition of~$\f{m}_{\bm{\Omega}}$ as the scalar curvature plus the map~$\f{m}$ defined in~\eqref{eq:moment_map_BiquardGauduchon}.
\begin{lemma}\label{lem:Futaki}
Let~$\alpha$ be a Higgs field. For any~$X\in\f{h}_\alpha$,
\begin{equation}
\m{F}_\alpha(X)=\m{F}_0(X).
\end{equation}
\end{lemma}

\begin{proof}
The aforementioned decomposition of~$\f{m}_{\bm{\Omega}}$ induces a decomposition of~$\m{F}_\alpha$ as
\begin{equation}
\begin{split}
\m{F}_\alpha(X)=&\int_Mh(X,\omega)\left(s(\omega)-\hat{s}\right)\frac{\omega^n}{n!}+\left\langle\f{m}\left(J,\alpha(\omega)\right),h(X,\omega)\right\rangle_{L^2(\omega)}=\\
=&\m{F}_0(X)+\frac{1}{2}\int_M\mrm{d}^cf\left(\m{L}_{X^\omega_h}\alpha(\omega)\right)\frac{\omega^n}{n!}.
\end{split}
\end{equation}
We must show that the integral of~$\mrm{d}^cf\left(\m{L}_{X^\omega_h}\alpha(\omega)\right)$ vanishes. We will actually show that the function~$\mrm{d}^cf\left(\m{L}_{X^\omega_h}\alpha(\omega)\right)$ is itself zero.

Taking the derivative of the spectral function~$f$, we have
\begin{equation}
\mrm{d}^cf\left(\m{L}_{X^\omega_h}\alpha(\omega)\right)=-2\Re\,\mrm{Tr}\left(f'(\bar{\alpha}(\omega)\alpha(\omega))\bar{\alpha}(\omega)\left(\m{L}_{X^\omega_h}\alpha(\omega)\right)J^\transpose\right).
\end{equation}
As~$X\in\f{h}_\alpha$, Lemma~\ref{lem:h_alpha} tells us that~$\m{L}_{X^\omega_h}\alpha(\omega)$ commutes with~$J$; since~$\bar{\alpha}(\omega)$ anti-commutes with~$J$ while~$f'(\bar{\alpha}(\omega)\alpha(\omega))$ commutes with it,
\begin{equation}
\mrm{Tr}\left(f'(\bar{\alpha}(\omega)\alpha(\omega))\bar{\alpha}(\omega)\left(\m{L}_{X^\omega_h}\alpha(\omega)\right)J^\transpose\right)=0
\end{equation}
hence~$\mrm{d}^cf\left(\m{L}_{X^\omega_h}\alpha(\omega)\right)$ vanishes.
\end{proof}

\subsection{Compatible test configurations}\label{sec:stability}

In this section, we will consider K\"ahler manifolds~$(M,J,\omega)$ and their test configurations, defined as in~\cite{DervanRoss_Kstab}. The discussion can be readily adapted to the algebraic situation, where one considers polarized varieties instead of K\"ahler manifolds.

Our goal is to establish a stability condition that characterizes the existence of solutions to the real moment map equation in~\eqref{eq:complessificate_intro}. More precisely, let~$\alpha$ be a Higgs field, a horizontal section of~$\m{E}\to[\omega]$, that solves the complex moment map equation~$\m{D}^*\alpha=0$. 
Lemma~\ref{lem:Futaki} suggests that to characterize the existence of a K\"ahler form in~$[\omega]$ solving the real moment map equation in terms of an analogue of K-stability, we can use the usual notion of test configurations and Donaldson-Futaki weight. We should however consider just the subset of test configurations that are, in a suitable sense, \emph{compatible} with the Higgs field, reflecting the fact that~$\mrm{Aut}_\alpha(M)$ is, in general, strictly contained in~$\mrm{Aut}_0(M)$.

We now describe a natural notion of compatibility between test configurations and Higgs fields. First, let us recall the notion of a test configuration from~\cite{DervanRoss_Kstab}.
\begin{definition}
Let~$[\omega]$ be a K\"ahler class on the compact complex manifold~$M$. A test configuration for~$(M,[\omega])$ is a normal K\"ahler space~$(\m{M},\Omega)$ together with a~$\bb{C}^*$-action and a flat surjective map~$\pi:\m{M}\to\bb{C}$ such that
\begin{enumerate}
\item~$\pi$ is~$\bb{C}^*$-equivariant, for the standard~$\bb{C}^*$-action on~$\bb{C}$;
\item the K\"ahler form~$\Omega$ is~$\bb{S}^1$-invariant, and the~$\bb{C}^*$-action preserves the Bott-Chern class of~$\Omega$;
\item there is a~$\bb{C}^*$-equivariant biholomorphisms~$\varphi$ from~$\pi^{-1}(\bb{C}^*)$ to~$M\times\bb{C}^*$ such that, for all~$t\in\bb{C}^*$,
$[\Omega_{\restriction\pi^{-1}(t)}]=[\varphi_t^*\omega]$.
\end{enumerate}
\end{definition}
For any such test configuration, we will denote by~$M_t$ the fibre over~$t\in\bb{C}$, and we let~$\omega_t$ be the restriction of~$\Omega$ on~$M_t$. We can assume for simplicity that~$(M_1,[\omega_1])=(M,[\omega])$, so that if~$\gamma:\bb{C}^*\to\mrm{Aut}(\m{M})$ denotes the~$\bb{C}^*$-action we have that for all~$t\not=0$,~$\gamma_t$ is a biholomorphism between~$M$ and~$M_t$ such that~$[\gamma_t^*\omega_t]=[\omega]$. Moreover,~$\gamma$ induces a group of automorphisms of the central fibre.

Assume now that~$\alpha$ is a Higgs field for~$(M,\omega)$, and let~$(\m{M},\Omega)$ be a test configuration. We can use the~$\bb{C}^*$-action to induce a first-order deformation of the complex structure on each fibre of~$\m{M}\to\bb{C}^*$. For~$t\not=0$, let
\begin{equation}\label{eq:alpha_testconfig}
\alpha_t=\gamma_t^{-1}{}^*\left(\alpha(\gamma_t^*\omega_t)\right)\in\m{A}^{0,1}(T^{1,0}M_t).
\end{equation}
As~$\alpha$ is a Higgs field and~$[\gamma_t^*\omega_t]=[\omega]$,~$\alpha(\gamma_t^*\omega_t)$ is well-defined and compatible with~$\gamma_t^*\omega_t$, so~$\alpha_t$ is a first-order deformation of~$M_t$ that is compatible with~$\omega_t$. Hence, we obtain a family of~$(0,1)$-forms on~$\m{M}^*\coloneqq \pi^{-1}(\bb{C}^*)$
\begin{equation}\label{eq:testconf_def}
\m{A}^*\coloneqq \set*{\alpha_t\tc t\in\bb{C}^*}\in\m{A}^{0,1}(T^{1,0}\m{M}^*).
\end{equation}
\begin{rmk}
Notice that~$\m{A}^*$ takes values in the vertical tangent bundle of~$\m{M}^*$, and vanishes on non-vertical vectors: in other words,~$\m{A}^*$ is a first-order deformation of the manifold~$\m{M}^*$ compatible with the map~$\pi:\m{M}^*\to\bb{C}^*$ and the K\"ahler form~$\Omega$. Notice also that~$\m{A}^*$ is an integrable deformation of~$\m{M}^*$, as~$\alpha$ is integrable and the~$1$PS~$\gamma$ is holomorphic. The~$(0,1)$-forms satisfying these properties have been studied in the deformation theory for holomorphic fibrations. We refer the interested reader to~\cite{Horikawa_deformations1, Horikawa_deformations2} for more details on the deformation theory of fibrations. 
\end{rmk}

\begin{definition}\label{def:compatible_testconfig}
Let~$(\m{M},\Omega)$ be a test configuration for~$(M,\omega)$, and for a Higgs term~$\alpha$ consider the first-order deformation~$\m{A}^*$ of~$(\m{M}^*,\Omega_{|\m{M}^*})$ defined by~$\alpha$ and the~$\bb{C}^*$-action as in~\eqref{eq:testconf_def},~\eqref{eq:alpha_testconfig}. We say that~$(\m{M},\Omega)$ is \emph{compatible} with~$\alpha$ if~$\m{A}^*$ can be extended to a first-order deformation~$\m{A}$ of~$(\m{M},\Omega)$.
\end{definition}

This definition is motivated by Lemma~\ref{lem:Futaki}: we want to consider test-configurations for which the~$\bb{C}^*$-action induced on the central fibre is compatible with the Higgs field~$\alpha$. The next result shows that Definition~\ref{def:compatible_testconfig} achieves this goal. 
\begin{lemma}\label{lem:smooth_compatible_testconf}
Let~$(\m{M},\Omega)$ be a smooth test configuration with a smooth central fibre, compatible with a Higgs field~$\alpha$. Let~$\gamma:\bb{C}^*\to\mrm{Aut}(\m{M})$ be the~$\bb{C}^*$-action, and let~$\alpha_0$ be the extension of~$\m{A}^*$ to the central fibre. Then the automorphisms induced by~$\gamma$ on the central fibre belong to~$\mrm{Aut}_{\alpha_0}(M_0)$.
\end{lemma}
If~$M_0$ is just a reduced K\"ahler space, the same conclusion holds on the smooth locus of the central fibre. It is not clear what behaviour we might expect at the singular points of~$M_0$, even if the notion of compatibility still makes sense when~$M_0$ has singularities.

\begin{proof}[Proof of Lemma~\ref{lem:smooth_compatible_testconf}]
Extend~$\alpha$ to a first-order deformation~$\m{A}^*$ of the K\"ahler fibration~$(\m{M}^*,\Omega_{|\m{M}^*})\to\bb{C}^*$, as in~\eqref{eq:testconf_def},~\eqref{eq:alpha_testconfig}. Consider the bundle of compatible pairs~$\m{E}_t\to[\omega_t]$, defined in Section~\ref{sec:bundle_class}. The~$(0,1)$-form~$\alpha_t$ defined in~\eqref{eq:alpha_testconfig} is compatible with~$\omega_t$, so we can extend it to a horizontal section of~$\m{E}_t$, which we still denote by~$\alpha_t$. Since horizontal extensions are unique, it is easy to show (as in Lemma~\ref{lemma:hol_pullbak_connection}) that for~$t\not=0$ and every~$\omega'\in[\omega_t]$
\begin{equation}\label{eq:alpha_extension_C*}
\alpha_t(\omega')=\gamma_t^{-1}{}^*\left(\alpha(\gamma_t^*\omega')\right).
\end{equation}
As~$\gamma$ is a group homomorphism, we can deduce from~\eqref{eq:alpha_extension_C*} that, for every~$t,s\in\bb{C}^*$ and every~$\omega'\in[\omega_{ts}]$
\begin{equation}\label{eq:testconfig_Higgs}
\alpha_{ts}(\omega')=\gamma_t^{-1}{}^*\left(\alpha_s(\gamma_t^*\omega')\right).
\end{equation}
Consider~\eqref{eq:testconfig_Higgs} for~$\omega'=\omega_{ts}$. As~$\Omega$ is a (smooth) K\"ahler form on~$\m{M}$ and~$\omega_t\coloneqq \Omega_{\restriction M_t}$, we have~$\omega_0=\lim_{s\to 0}\omega_s$. Moreover, our hypothesis is that~$\alpha_0=\lim_{s\to 0}\alpha_s$, so we can take the limit for~$s\to 0$ and we find
\begin{equation}
\alpha_0(\omega_0)=\gamma_t^{-1}{}^*\left(\alpha_0(\gamma_t^*\omega_0)\right),
\end{equation}
so that~$\gamma(\bb{C}^*)\subseteq\mrm{Aut}_{\alpha_0}(M_0)$, see Remark~\ref{rmk:aut_alpha}.
\end{proof}
As an example of Definition~\ref{def:compatible_testconfig}, it might be useful to check whether a product test configuration is compatible with a first-order deformation~$\alpha$. Assume that~$(M,\omega)$ admits a~$\bb{C}^*$-action,~$\set*{\psi_s\in\mrm{Aut}(M,J)\tc s\in\bb{C}^*}$, generated by~$X\in\f{h}_0$. We define on the product~$\m{M}=M\times\bb{C}$ the K\"ahler form~$\Omega=p_1^*\omega'+p_2^*\omega_{\mrm{FS}}$, where~$\omega'$ is any~$\bb{S}^1$-invariant form in~$[\omega]$ and~$\omega_{\mrm{FS}}$ is the Fubini-Study form induced on~$\bb{C}$ by the inclusion in~$\bb{P}^1$. The~$\bb{C}^*$-action on~$M$ induces an action on~$\m{M}$ by~$\gamma_s(z,t)\coloneqq (\psi_s(z),st)$. The action induced on the central fibre is just the~$\bb{C}^*$-action that we started with. Lemma~\ref{lem:smooth_compatible_testconf} shows that a necessary condition for~$M\times\bb{C}$ to be compatible with the Higgs field~$\alpha$ is that~$X\in\f{h}_\alpha$. On the other hand, if~$X\in\f{h}_\alpha$ then~$\alpha_t=\alpha$ for every~$t\not=0$, where~$\alpha_t$ is defined by~\eqref{eq:alpha_extension_C*}, so~$\alpha_t$ can be extended to the central fibre.

Lemma~\ref{lem:smooth_compatible_testconf} provides an obstruction to compatibility between an arbitrary test configuration and a first-order deformation of the complex structure. In general, it seems difficult to check whether a given test configuration is compatible with a Higgs term or not. In the polarized case there is however a class of test configurations for which the compatibility condition takes a simpler form. For a polarized manifold~$(X,L)$ and a closed subscheme~$Z\subset X$, the deformation to the normal cone of~$Z$ is a test configuration obtained by blowing up~$Z\times\set{0}$ in~$X\times\bb{A}^1$, see \cite{RossThomasHMcriterion}. Let~$\alpha$ be a Higgs field on~$X$, defining a class~$\eta\in H^1(X,\m{E}_L)$; as~$\eta$ defines a fibrewise deformation of~$X\times\bb{A}^1$, the deformation to the normal cone of~$Z$ is compatible with~$\alpha$ (or~$\eta$) if the deformation preserves~$Z$. This description of compatible test configurations can be adapted to the more general class of (semi-)test configurations induced by blowups of flag ideals studied in \cite{Odaka_blowup}.

The definition of K-stability, uniform K-stability and K-semistability for system~\eqref{eq:system_mm_implicit} can be readily adapted from the analogous notions in the cscK setting; we just have to check the conditions on the test configurations that are compatible with the Higgs field. A version of the Yau-Tian-Donaldson conjecture for our system of equations could be
\begin{conj}\label{conj:Kstab}
Let~$(M,\omega_0)$ be a compact K\"ahler manifold, and let~$\alpha$ be a Higgs field satisfying~$\m{D}^*\alpha=0$. Then, system~\eqref{eq:system_mm_implicit} admits a solution~$\omega\in[\omega_0]$ if and only if~$(M,[\omega_0])$ is K-stable with respect to test configurations that are compatible with~$\alpha$.
\end{conj}
This statement should be interpreted only as a guiding principle for future research. For example, it might be the case that K-stability will guarantee the existence of solutions of the real moment map equation only for small Higgs fields. In fact, we should expect this when the function defining our Calabi Ansatz metric~\eqref{eq:metrica_TJ} is defined only on a subset of~$\cotJ$, rather than the full space. For example, if the spectral function defining the metric on~$\cotJ$ is~\eqref{eq:hyperk_potential}, the equation degenerates when the eigenvalues of~$\bar{\alpha}(\omega)\alpha(\omega)$ tend to~$1$.

Conjecture~\ref{conj:Kstab} predicts, among other things, that on any cscK manifold~$(M,\omega)$ there should be a solution of equation~\eqref{eq:mm_real_explicit}, for any choice of a (small enough) Higgs field. We will see an example of this phenomenon in the context of toric manifolds, in Section~\ref{sec:toric}.

\begin{rmk}
This stability condition does not depend on the choice of the spectral function used to define the metric~\eqref{eq:metrica_TJ} on~$\cotJ$. This phenomenon is to be expected, as this version of K-stability should still be similar to a GIT stability notion: while the moment map equations~\eqref{eq:system_mm_implicit} depend on the metric~$\bm{\Omega}$ on~$\cotJ$, the stability of any point~$(J,\alpha)\in\cotJ$ depends instead on the action of~$\m{G}$ and the polarization of~$\cotJ$, or the class~$[\bm{\Omega}]$, rather than the particular choice of metric in~$[\bm{\Omega}]$. 

There is also a similar phenomenon in the context of Higgs bundles. The paper~\cite{Mundet_Hitchin_corr} studies a variety of K\"ahler reductions, that generalize the Higgs bundle equations. These reductions are defined on the product~$\m{A}\times\m{S}$ of a space of connections on~$E$ with the space of sections of a second bundle. The equations depend on the choice of a fibrewise metric on this bundle, that is used to define the metric on~$\m{A}\times\m{S}$. It turns out, however, that the existence of solutions to these equations is equivalent to an algebraic stability condition that does not depend on this choice, see~\cite[Theorem~$2.19$]{Mundet_Hitchin_corr}.
\end{rmk}

\section{The toric case}\label{sec:toric}

Assume that~$(M,\omega)$ is a toric manifold, and denote by~$P\subset\bb{R}^n$ its moment polytope. The set~$M^\circ\subset M$ where the~$\bb{T}$-action is free is identified with~$P\times\bb{R}$, and in the system of real coordinates~$\left(\vec{y},\vec{w}\right)$ on~$P\times\bb{R}^n$ the K\"ahler form~$\omega$ takes the form~$\omega=\sum_a\mrm{d}y^a\wedge\mrm{d}w^a$.

The complex structure is described by the Hessian of a convex function~$u\in\m{C}(P)\cap\m{C}^\infty(P^\circ)$, commonly called the \emph{symplectic potential} of the metric~$g_J$. The symplectic potentials have a prescribed singular behaviour at the boundary of the polytope: assume that~$P$ is described by the set of linear inequalities~$\ell_i(x)\geq 0$ for~$i=1,\dots,r$; a convex function~$u$ is a symplectic potential if and only if~$u(x)-\sum_{i=1}^r\ell_i(x)\log\ell_i(x)$ is smooth up to the boundary of the polytope. This behaviour is known as \emph{Guillemin's boundary conditions}.

Similarly, a first-order deformation~$\alpha$ of the complex structure can be identified with the product matrix~$D^2(u)^{-1}H$, for some symmetric matrix-valued function~$H$; the deformation will be integrable if and only if~$H=D^2(h)$ for some~$h\in\m{C}^\infty(P)$, see the proof of Proposition~$1.14$ in~\cite{ScarpaStoppa_symplectic}. In this section we will only focus on the real moment map equation~\eqref{eq:mm_real_definition}, as there are no integrable torus-invariant solutions of the complex moment map~$\m{D}^*\alpha=0$, see Corollary~$7.2$ in~\cite{ScarpaStoppa_symplectic}.

We start by characterizing the solutions to the horizontal extension equation~\eqref{eq:hor_lift}.
\begin{lemma}\label{lem:toric_horlift}
Let~$u_t$ be a path of symplectic potentials on~$P$, corresponding to a path of K\"ahler metrics~$\omega_t$, and for a fixed~$h\in\m{C}^\infty(P)$ consider the path of first-order deformations of the complex structure~$D^2u_t^{-1}D^2h$, corresponding to~$\alpha_t\in\m{A}^{0,1}(T^{1,0}M)$. Then~$\alpha_t$ is a horizontal lift of~$\omega_t$.
\end{lemma}
\begin{proof}
Recall that we use complex coordinates~$\left(z^a=x^a+\I w^a\right)_{1\leq a\leq n}$ on~$M^\circ$, and real coordinates~$(y^a)_{1\leq a\leq n}$ on~$P$. The K\"ahler metrics~$\omega_t$ can be written in~$M^\circ$ as~$\omega_t=\I\diff\bdiff\varphi_t$ for some local torus-invariant potential; if we define~$v_t=\varphi_t/4$, then
\begin{equation}
\omega=\I\,\diff_{z^a}\diff_{\bar{z}^b}\varphi_t\,\mrm{d}z^a\wedge\mrm{d}\bar{z}^b=\I\,\diff_{x^a}\diff_{x^b}v_t\,\mrm{d}z^a\wedge\mrm{d}\bar{z}^b
\end{equation}
and~$u_t$ is the Legendre dual of~$v_t$. Moreover, in these coordinates the moment map for the torus action with respect to~$\omega_t$ is given by~$\mu_t\coloneqq \diff_xv_t$.

With this notation,~$\alpha_t$ is written in complex coordinates as
\begin{equation}
(\alpha_t)\indices{_{\bar{b}}^a}=u_t^{bc}\diff_{y^c}\diff_{y^a}h=\diff_{x^b}\left(v_t^{ac}\diff_{x^c}(h\circ\mu_t)\right)=\left(\bdiff\nabla^{1,0}_t(h\circ\mu_t)\right)\indices{_{\bar{b}}^a}.
\end{equation}
Taking the derivative in~$t$, we find
\begin{equation}
(\diff_t\alpha_t)\indices{_{\bar{b}}^a}=\bdiff\left(\alpha_t(\diff\varphi'_t)^{\sharp_t}\right)\indices{_{\bar{b}}^a}+\diff_{\bar{z}^b}\nabla_t^{a}\left(\diff_t(h\circ\mu_t)-\left\langle\diff\varphi'_t,\bdiff(h\circ\mu_t)\right\rangle_t\right).
\end{equation}
Comparing this with~\eqref{eq:hor_lift}, we see that to prove the Lemma it will be enough to show
\begin{equation}\label{eq:toric_horlift}
\diff_t(h\circ\mu_t)=\left\langle\diff\varphi'_t,\bdiff(h\circ\mu_t)\right\rangle_t.
\end{equation}
As~$\mu_t=\diff_xv_t=\diff_x\varphi_t/4$, the left-hand side of~\eqref{eq:toric_horlift} is
\begin{equation}
\diff_t(h\circ\mu_t)=\sum_a\diff_{y^a}h\diff_t\mu^a_t=\frac{1}{4}\sum_a\diff_{y^a}h\diff_{x^a}\varphi'_t.
\end{equation}
For the right-hand side of~\eqref{eq:toric_horlift} instead we have
\begin{equation}
\left\langle\diff\varphi'_t,\bdiff(h\circ\mu_t)\right\rangle_t=\diff_{z^a}\varphi'_tv_t^{ab}\diff_{\bar{z}^b}(h\circ\mu_t)=\frac{1}{4}\diff_{x^a}\varphi'_tv_t^{ab}\diff_{x^b}(h\circ\mu_t)=\frac{1}{4}\sum_a\diff_{x^a}\varphi'_t\diff_{y^a}h
\end{equation}
where for the last equality we have used the fact that~$v_t(x)$ and~$u_t(y)$ are Legendre duals of each others.
\end{proof}

The scalar curvature of~$g_J$ has a simple expression in terms of the symplectic potential
\begin{equation}
s(\omega,J)(\vec{y})=-\frac{1}{4}\diff_a\diff_b\mrm{Hess}(u)^{ab}.
\end{equation}
Letting~$G_{ab}\coloneqq \diff_a\diff_b u$,~$H_{ab}\coloneqq \diff_a\diff_b h$ and~$\alpha\coloneqq G^{-1}H$, the real moment map equation~\eqref{eq:mm_real_explicit} takes the form
\begin{equation}\label{eq:toric_realeq}
-\diff_a\diff_b\left(\left(1+2\,f'(\bar{\alpha}\alpha)\bar{\alpha}\alpha\right)G^{-1}\right)^{ab}=A_0.
\end{equation}
\begin{rmk}\label{eq:toric_posdef}
The matrix~$\left(1+2\,f'(\bar{\alpha}\alpha)\bar{\alpha}\alpha\right)G^{-1}$ is positive-definite; to see this, fix a point~$p\in P$ a system of linear coordinates on~$\bb{R}^n$ such that, at~$p$,~$G=\mathbb{1}$,~$\bar{H}H=\mrm{diag}(\lambda_1,\dots,\lambda_n)$. Recall that~$f(\vec{\lambda})=\sum_ak(\lambda_a)$ for a convex, non-decreasing function~$k$. Then, at the point~$p$ one has
\begin{equation}
\left(1+2\,f'(\bar{\alpha}\alpha)\bar{\alpha}\alpha\right)G^{-1}=\mrm{diag}\left(1+2\,\lambda_1\,k'(\lambda_1),\dots,1+2\,\lambda_n\,k'(\lambda_n)\right)
\end{equation}
and each of these eigenvalues is positive.
\end{rmk}
The next result is inspired by an integration by parts formula of Donaldson, see Lemma~$3.3.5$ in~\cite{Donaldson_stability_toric}, and its proof can be obtained with minor modifications from the proof of an analogous result in~\cite{ScarpaStoppa_symplectic}, Lemma~$5.1$.
\begin{lemma}\label{lemma:intbyparts_realmm}
For a Delzant polytope~$P\subset\bb{R}^n$, fix a symplectic potential~$u$ satisfying Guillemin's boundary conditions and a first-order deformation of the complex structure~$\alpha$. For any function~$v\in\m{C}^0(P)\cap\m{C}^\infty(P^\circ)$ that is either convex or smooth on~$P$ we have
\begin{align*}
&\int_P\mrm{Tr}\left(\left(1+2\,f'(\bar{\alpha}\alpha)\bar{\alpha}\alpha\right)G^{-1}D^2v\right)\mrm{d}\mu=\\&=\int_Pv\left(\left(1+2\,f'(\bar{\alpha}\alpha)\bar{\alpha}\alpha\right)G^{-1}\right)^{ab}_{,ab}\mrm{d}\mu+\int_{\diff P}\!v\,\mrm{d}\sigma.
\end{align*}
\end{lemma}
\begin{rmk}
Compare Lemma~\ref{lemma:intbyparts_realmm} with Proposition~\ref{prop:Futaki_vanish}. Lemma~\ref{lemma:intbyparts_realmm} implies that a necessary condition for the existence of solutions to the real moment map equation is that for every affine-linear function~$u\in\m{C}^\infty(P)$
\begin{equation}
\int_{\diff P}u\mrm{d}\sigma-\int_PA_0u\mrm{d}\mu=0
\end{equation}
which amounts to saying that the Futaki invariant must vanish on torus-invariant holomorphic vector fields, which are all contained in~$\f{h}_\alpha$.
\end{rmk}
As a corollary of Lemma~\ref{lemma:intbyparts_realmm}, if~$u$ and~$\alpha$ solve equation~\eqref{eq:toric_realeq} we see that, for any convex function~$v\in\m{C}^0(P)\cap\m{C}^\infty(P^\circ)$,
\begin{equation}
\int_{\diff P}\!v\,\mrm{d}\sigma-\int_Pv\,A_0\mrm{d}\mu=\int_P\mrm{Tr}\left(\left(1+2\,f'(\bar{\alpha}\alpha)\bar{\alpha}\alpha\right)G^{-1}D^2v\right)\mrm{d}\mu.
\end{equation}
As~$\left(1+2\,f'(\bar{\alpha}\alpha)\bar{\alpha}\alpha\right)G^{-1}$ is a positive-definite matrix, there is some constant~$\lambda>0$ such that
\begin{equation}
\int_{\diff P}\!v\,\mrm{d}\sigma-\int_Pv\,A_0\mrm{d}\mu>\lambda\int_{\diff P}v\mrm{d}\sigma
\end{equation}
so that~$M$ is uniformly (toric-)$K$-stable, see~\cite[\S$5$]{Li_uniformKstab}. In other words, K-stability is a necessary condition for the existence of torus-invariant solutions of~\eqref{eq:mm_real_definition}. On the other hand, our conjecture in Section~\ref{sec:stability} says that (toric) K-stability should imply the existence of solutions of our equation. We can verify that this is indeed the case, at least for small enough deformations of the complex structure. To do so, we first need some properties of the variational characterization of equation~\eqref{eq:toric_realeq}.
\begin{prop}\label{prop:toric_convexity}
The real moment map equation~\eqref{eq:toric_realeq} is the Euler-Lagrange equation of the functional
\begin{equation}
\m{HK}(u)=\int_{\diff P}u\mrm{d}\sigma-\int_PA_0u\mrm{d}\mu-\int_P\log\det\,D^2u\,\mrm{d}\mu+\int_Pf(G^{-1}\bar{H}G^{-1}H)\mrm{d}\mu.
\end{equation}
This functional is convex along linear paths of symplectic potentials.
\end{prop}
\begin{proof}
The first part follows from Theorem~\ref{thm:variational} and Lemma~\ref{lem:toric_horlift}, as the K-energy in the toric setting is precisely~$\m{M}(u)=\int_{\diff P}u\mrm{d}\sigma-\int_PA_0u\mrm{d}\mu-\int_P\log\det\,D^2u\,\mrm{d}\mu$. The proof of the convexity is completely analogous to the proof of the ellipticity of Theorem~\ref{thm:intro_linearizzazione} in Section~\ref{sec:linearization}.
\end{proof}
This result in particular implies uniqueness of solutions of~\eqref{eq:toric_realeq} (up to adding affine functions), and shows that the linearization of~\eqref{eq:toric_realeq} around a solution is surjective on the space of functions that are~$L^2$-orthogonal to affine functions. As the linearization is self-adjoint and its kernel is exactly the affine functions on~$P$, Proposition \ref{prop:toric_convexity} implies the existence of solutions to~\eqref{eq:toric_realeq}, for arbitrarily small~$H$, under the assumption of uniform toric K-stability of~$M$. Using an observation from \cite{AnMinLi_prescribed_toric} (after~\cite{ChenCheng_estimates}), it is possible to provide a quantitative bound on~$H$ that guarantees the existence of solutions to~\eqref{eq:toric_realeq}. 

\begin{prop}\label{prop:toric_existence}
Assume that~$P$ is a uniformly K-stable polytope. There is a constant~$C$ such that, for any complex symmetric matrix~$H$ satisfying~$\norm{H}_{\m{C}^2}<C$, there exists a symplectic potential~$u$ solving equation~\eqref{eq:toric_realeq}.
\end{prop}
The constant~$C$ can in principle be computed from the geometry of~$P$ and the estimates established in~\cite[Theorem~$1.2$]{ChenCheng_estimates}.
\begin{proof}
We consider the continuity method, for~$t\in[0,1]$
\begin{equation}\label{eq:toric_contmet}
\tag{$\star_t$}
-\left(\left(1+2tf'(\bar{\alpha}\alpha)\bar{\alpha}\alpha\right)G^{-1}\right)^{ab}_{,ab}=A_0.
\end{equation}
For~$t=0$, a solution in given by a cscK potential~$u_0$, that exists as we are assuming~$P$ to be uniformly K-stable. Openness of the continuity method is guaranteed by the convexity result of Proposition~\ref{prop:toric_convexity}, as we briefly outlined above.

\noindent{\textit{A priori estimates.}}\quad Define the function~$
A_t=A_0+2t\left(f'(\bar{\alpha}\alpha)\bar{\alpha}\alpha\,G^{-1}\right)^{ab}_{,ab}$ so that the continuity method can alternatively be written as
\begin{equation}\label{eq:toric_contmet_Abreu}
-u^{ab}_{,ab}=A_t.
\end{equation}
As~$P$ is uniformly K-stable, if the~$\m{C}^2$-norms of~$\alpha$ and~$\alpha G^{-1}$ are small enough, then~$(P,A_t)$ will also be uniformly K-stable. More explicitly, recall that~$(P,A)$ is uniformly K-stable if
\begin{equation}
\m{L}_{A_0}(v)>\lambda\int v\mrm{d}\sigma
\end{equation}
for every normalized convex function~$v\in\m{C}^\infty(P)$. Then, if~$(P,A_0)$ is uniformly K-stable and~$2\norm{f'(\bar{\alpha}\alpha)\bar{\alpha}\alpha G^{-1}}_{\m{C}^2}<C<\lambda A_0(1-\lambda)^{-1}$, for every normalized~$v$ we have
\begin{equation}
\begin{split}
\int_{\diff_P}v\mrm{d}\sigma-&\int_PA_tv\mrm{d}\mu>\lambda\int_{\diff_P}v\mrm{d}\sigma-\int_P2t\left(f'(\bar{\alpha}\alpha)\bar{\alpha}\alpha\,G^{-1}\right)^{ab}_{,ab}v\mrm{d}\mu>\\
>&\lambda\int_{\diff_P}v\mrm{d}\sigma-\frac{t\,C}{A_0}\int_PA_0v\mrm{d}\mu>\int_{\diff_P}v\mrm{d}\sigma\left(\lambda-\frac{C}{A_0}(1-\lambda)\right)
\end{split}
\end{equation}
so that~$(P,A_t)$ is uniformly~$\lambda'$-stable, for~$\lambda'=\lambda-C(1-\lambda)A_0^{-1}$.

This implies a priori estimates along the continuity method~\eqref{eq:toric_contmet_Abreu} in terms of the norm of~$f'(\bar{\alpha}\alpha)\bar{\alpha}\alpha G^{-1}$. Indeed, the results in~\cite[\S$4$]{AnMinLi_prescribed_toric} show how solutions of Abreu's equation~$-u^{ab}_{ab}=A$ have a priori estimates in terms of~$\sup A$ and a uniform stability threshold for~$(P,A)$. If we write the solution~$u_t$ of \eqref{eq:toric_contmet_Abreu} as~$u_t=u_0+v_t$ for some~$v_t\in\m{C}^\infty(P)$, then there is a constant~$C$ such that~$\norm{f'(\bar{\alpha}\alpha)\bar{\alpha}\alpha G^{-1}}_{\m{C}^2}<C$ implies a bound~$\norm{v_t}_{\m{C}^4}<C^*$.

As the norm of~$G^{-1}$ is estimated in terms of the norms of~$u_0$ and~$v_t$, this implies that, if~$\norm{f'(\bar{\alpha}\alpha)\bar{\alpha}\alpha}_{\m{C}^2}$ is small enough, then the~$\m{C}^4$-norm of the solutions of the continuity method~\eqref{eq:toric_contmet_Abreu} is uniformly bounded by~$C^*$. Recalling that~$\alpha=G^{-1}H$, we see that there is a constant~$C'$ such that~$\norm{H}_{\m{C}^2}<C'$ implies a priori bounds for solutions~$v_t$ of~\eqref{eq:toric_contmet}.

Of course the previous reasoning can also be used to give~$\m{C}^{4,\alpha}$-bounds on solutions of~\eqref{eq:toric_contmet}. Then, any sequence of solutions of~\eqref{eq:toric_contmet} will converge to a new solution, up to subsequences, and the usual regularity theory for elliptic equations shows that the continuity method is closed.
\end{proof}

\appendix
\section{K\"ahler geometry of the space of complex structures}\label{sec:geometry_TJ}

In this Section we recall the basic geometric properties of~$\m{J}$ and~$\cotJ$, and we will prove that the Calabi ansatz~\eqref{eq:metrica_TJ_f} gives a K\"ahler metric on~$\cotJ$. For a more detailed description of the K\"ahler structure of~$\m{J}$ we refer the reader to~\cite{ScarpaStoppa_hyperk_reduction}.

First, consider the space~$\m{J}$. In a system of Darboux coordinates for~$\omega$, the symplectic form is identified with the matrix~$\Omega_0=\begin{pmatrix}0 & \bb{1}\\-\bb{1} &0\end{pmatrix}$, while any element of~$\m{J}$ is identified with a matrix~$J\in\bb{R}^{2n\times 2n}$ such that
\begin{equation}\label{eq:app_AC+}
J^2=-\bb{1},\  J^\transpose\Omega_0J=\Omega_0\mbox{ and }\Omega_0J>0.
\end{equation}
We denote with~$\m{AC}^+$ the space of~$2n\times 2n$ real matrices satisfying the conditions~\eqref{eq:app_AC+}, so that in a system of Darboux coordinates on~$U\subset M$, any element of~$\m{J}$ is identified with a section of~$U\times\m{AC}^+\to U$. From this point of view, the choice of a different system of Darboux coordinates on~$M$, corresponds to the action by conjugation of~$\mrm{Sp}(2n)$ on~$\m{AC}^+$. In other words,~$\m{J}$ is the space of section of a bundle~$F\to M$ whose fibres are isomorphic of~$\m{AC}^+$. This bundle trivializes over system of Darboux coordinates for~$\omega$, and the transition between two different trivialization is given by the action of~$\mrm{Sp}(2n)$.

The~$\mrm{Sp}(2n)$-action of~$\m{AC}^+$ is in fact transitive, and the stabilizer of any point is isomorphic to~$U(n)$, so that we can identify~$\m{AC}^+$ with~$\mrm{Sp}(2n)/U(n)$. This symmetric space has an alternative description, as it is isomorphic to Siegel's upper half space~$\f{H}$, the set of all~$n\times n$ complex matrices with a positive-definite imaginary part. The symplectic group acts on~$\f{H}$ by an analogue of the M\"obius transformations: for~$\begin{pmatrix}A & B\\ C& D\end{pmatrix}\in\mrm{Sp}(2n)$, the action is defined as
\begin{equation}
\begin{pmatrix}A & B\\ C& D\end{pmatrix}.Z=(AZ+B)(CZ+D)^{-1}.
\end{equation}
A~$\mrm{Sp}(2n)$-equivariant identification of~$\m{AC}^+$ and~$\f{H}$ is given by
\begin{equation}\label{eq:app_identification}
\begin{split}
\Psi\colon\m{AC}^+&\to\f{H}\\
P&\mapsto(\Omega_0P)^{-1/2}.\I\bb{1}
\end{split}
\end{equation}
Siegel's upper half space is a K\"ahler-Einstein manifold, and the~$\mrm{Sp}(2n)$-action preserves both the complex structure and the metric. Then,~$\m{AC}^+$ is also a K\"ahler manifold: the complex structure on~$T_J\m{AC}^+$ is given by~$A\mapsto JA$, while the metric is simply~$\langle A,B\rangle=\mrm{Tr}(AB)$. This shows that~$\m{J}$ is the space of sections of a bundle whose fibres carry a K\"ahler structure, invariant under the transition functions. By~\cite{Koiso_complex_structure},~$\m{J}$ inherits a K\"ahler structure from that of~$\m{AC}^+$, given by integrating the product of~$\m{AC}^+$ over~$M$.

A similar discussion can be made for~$T\m{J}$ and~$\cotJ$, that we identify using the fibrewise metric of~$F$. They inherit a complex structure from that of~$\m{AC}^+$, and a K\"ahler metric on~$T\m{AC}^+$ (or~$T\f{H}$) will define a K\"ahler metric on~$\cotJ$. The complex structure is given by
\begin{equation}\label{eq:app_AC+complexstructure}
\begin{split}
T_{J,A}T\m{AC}^+&\to T\m{AC}^+\\
(\dot{J},\dot{A})&\mapsto(J\dot{J},J\dot{A}+A\dot{J}).
\end{split}
\end{equation}
This can be seen by considering an element~$(J,A)$ of~$T\m{AC}$ as a first-order deformation~$J+\varepsilon A$ of~$J$; the complex structure~\eqref{eq:app_AC+complexstructure} is then simply given by the first two orders of the expansion of~$(J+\varepsilon A)(\dot{J}+\varepsilon\dot{A})$.

As for the metric, on~$T\m{AC}^+$ we consider spectral functions
\begin{equation}\label{eq:app_spectralfunc}
\begin{split}
f:T\m{AC}^+&\to\bb{R}\\
(J,A)&\mapsto f(A^2)
\end{split}
\end{equation}
and the Calabi ansatz metrics on~$T\m{AC}^+$ defined from the metric of~$\m{AC}^+$ and the complex Hessian of these spectral functions as
\begin{equation}
\omega_f=\pi^*\omega_{\m{AC}^+}+\mrm{d}\mrm{d}^cf.
\end{equation}
These K\"ahler metrics on~$T\m{AC}^+$ induce metrics on~$\cotJ$, as described in~\cite{Koiso_complex_structure}. Theorem~$3.2$ in~\cite{ScarpaStoppa_hyperk_reduction} shows that any Calabi ansatz metric on~$\cotJ$ of the type considered in~\eqref{eq:metrica_TJ} comes from this construction.
\begin{prop}\label{prop:convex_potential}
Let~$f(J,A)$ be a symmetric function of the eigenvalues of~$A^2$. If~$f$ is convex, then~$\mrm{d}\mrm{d}^cf$ is a positive-definite~$(1,1)$-form on~$T\m{AC}^+$.
\end{prop}
We will prove this using the~$\mrm{Sp}(2n)$-equivariant identification of~$\m{AC}^+$ with~$\f{H}$ of~\eqref{eq:app_identification}. Any~$A\in T_{-\Omega_0}\m{AC}^+$ can be written in block-matrix notation as
\begin{equation}
A=\begin{pmatrix}
X & Y\\ Y & -X
\end{pmatrix}\ \mbox{ for }X^\transpose=X,\ Y^\transpose=Y
\end{equation}
and~$\mrm{d}\Psi_{-\Omega_0}(A)=X+\I Y$. The eigenvalues of~$A^2$ coincide with those of
\begin{equation}
\mrm{d}\Psi_{-\Omega_0}(A)\,\overline{\mrm{d}\Psi_{-\Omega_0}(A)}=X^2+Y^2+\I(YX-XY)\in T_{\I\bb{1}}\f{H},
\end{equation}
and since~$\Psi$ is~$\mrm{Sp}(2n)$-equivariant, we can use the action to pull back~$f$ to~$\f{H}$:
\begin{equation}\label{eq:potential_Siegel}
f\left(\mrm{d}\Psi^{-1}(Z,S)\right)=k\left(\lambda_1,\dots,\lambda_n\right)\mbox{ for }\lambda_a=\lambda_a\left(\Im(Z)^{-\frac{1}{2}}S\Im(Z)^{-1}\bar{S}\Im(Z)^{-\frac{1}{2}}\right).
\end{equation}
By a slight abuse of notation, we will still denote with~$f$ the function defined on~$T\f{H}$ by~\eqref{eq:potential_Siegel}.
\begin{rmk}
Notice that~$\mrm{d}\mrm{d}^cf$ vanishes, when restricted to the zero-section of~$T\f{H}$. Indeed, if~$S=0$ and~$\dot{S}_1=\dot{S}_2=0$ then~$f(Z+s\dot{Z}_1+t\dot{Z}_2,S+s\dot{S}_1+t\dot{S}_2)$ is a constant, as the eigenvalues of
\begin{equation}
\Im(Z)^{-\frac{1}{2}}\left(S+s\dot{S}_1+t\dot{S}_2\right)\Im(Z)^{-1}\overline{\left(S+s\dot{S}_1+t\dot{S}_2\right)}\Im(Z)^{-\frac{1}{2}}
\end{equation}
vanish identically. Hence,~$\mrm{d}\mrm{d}^cf_{Z,0}\left((Z_1,0),(Z_2,0)\right)=0$.
\end{rmk}
\begin{proof}[Proof of Proposition~\ref{prop:convex_potential}]
As~$f\in\m{C}^\infty(T\f{H})$ is~$\mrm{Sp}(2n)$-invariant, it will be enough to show that~$\mrm{d}\mrm{d}^cf$ is a positive~$(1,1)$-form on~$T_{Z,S}\left(T\f{H}\right)$ for~$Z=\I\bb{1}$. In other words, we should check that, for any vector~$v\in T_{\I\bb{1},S}T\f{H}$ (that defines a constant vector field on~$T\f{H}$),~$\left(\mrm{d}\mrm{d}^cf\right)_{(\I\bb{1},S)}(v,Jv)>0$. By definition,
\begin{equation}
\mrm{d}\mrm{d}^cf(v,Jv)=v(v(f))+Jv(Jv(f))
\end{equation}
so it is sufficient to check that~$v(v(f))>0$ for any constant vector field. This second derivative of~$f$ can be computed by the same techniques of the proof of Theorem~\ref{thm:intro_linearizzazione}. Alternatively, an expression for~$v(v(f))$ can be found in~\cite[\S$5$]{LewisSendov_Hessian}. For ease of notation, let us assume that the eigenvalues of~$S\bar{S}$ are all distinct. For every~$(Z,S)\in T\f{H}$, let
\begin{equation}
A(Z,S)=\Im(Z)^{-\frac{1}{2}}S\Im(Z)^{-1}\bar{S}\Im(Z)^{-\frac{1}{2}},
\end{equation}
so that~$A(\I\bb{1},S)=S\bar{S}$, and for~$v\in T_{\I\bb{1},S}T\f{H}$ let~$B$ be the Hermitian matrix that satisfies
\begin{equation}
A\left((\I\bb{1},S)+t\,v\right)=S\bar{S}+t\,B+O(t^2).
\end{equation}
There is a unitary matrix~$U$ such that~$S\bar{S}=U\mrm{diag}\left(\lambda_1,\dots,\lambda_n\right)U^*$; if we let~$B'=U^*BU$ then
\begin{equation}\label{eq:vvf}
v(v(f))=\sum_{a,b}\diff_a\diff_bf\,B'_{aa}\,B'_{bb}+\sum_{a\not=b}\frac{\diff_af-\diff_bf}{\lambda_a-\lambda_b}\abs*{B'_{ab}}^2.
\end{equation}
As~$f$ is convex and symmetric, it can be checked by restricting to the~$\lambda_a,\lambda_b$-plane that
\begin{equation}
\frac{\diff_af-\diff_bf}{\lambda_a-\lambda_b}\geq 0,
\end{equation}
so~\eqref{eq:vvf} is positive.
\end{proof}

\addcontentsline{toc}{section}{References}
\bibliographystyle{amsalpha}
\bibliography{../biblio_HCSCK}

\end{document}